\documentclass[10.5pt, a4paper]{amsart}
\usepackage{amssymb, amsmath, amscd, bm, amsthm, pdfpages, setspace, hyperref, mathtools, subcaption, graphicx, cite, xfrac, empheq}

\def\@Rref#1{\hbox{\rm \ref{#1}}}
\def\Rref#1{\@Rref{#1}}
\theoremstyle{plain}
\newtheorem{theorem}{Theorem}[section]

\newtheorem{assumption}[theorem]{Assumption}
\newtheorem{lemma}[theorem]{Lemma}

\theoremstyle{definition}
\newtheorem{definition}{Definition}[section]

\newtheorem{remark}[definition]{Remark}

\newcommand{\op}{\mathop{\mathrm{op}}\nolimits}

\newcommand{\cl}{\mathop{\mathrm{cl}}\nolimits}

\newcommand{\bop}{\mathop{\mathbf B}_{\op}\nolimits}
\newcommand{\bcl}{\mathop{\mathbf B}_{\cl}\nolimits}

\newcommand{\floor}{\mathop{\mathrm{floor}}\nolimits}

\begin{document}

\title{Self-triggered Stabilization of Contracting Systems under 
	Quantization}

\thispagestyle{plain}

\author{Masashi Wakaiki}
\address{Graduate School of System Informatics, Kobe University, Nada, Kobe, Hyogo 657-8501, Japan}
 \email{wakaiki@ruby.kobe-u.ac.jp}
 \thanks{This work was supported by JSPS KAKENHI Grant Number JP20K14362.}

\begin{abstract}
We propose self-triggered control schemes for nonlinear systems with
quantized state measurements.
Our focus lies on scenarios where both the controller and
the self-triggering mechanism 
receive only the quantized state at each sampling time.
We assume that the ideal closed-loop system without quantization or self-triggered
sampling is contracting.
Moreover, an upper bound on the 
growth rate of the open-loop system 
is assumed to be known.
We present two control schemes that achieve 
closed-loop stability 
without Zeno behavior.
The first scheme is implemented under 
logarithmic quantization and uses the quantized state
for the threshold in the triggering condition.
The second one is a joint design of zooming quantization and 
self-triggered sampling, where 
the adjustable zoom parameter for quantization changes 
based on inter-sampling times and is also used for the threshold of
self-triggered sampling. In both schemes,
the self-triggering mechanism predicts
the future state 
from the quantized data
for the computation of the next sampling time.
We employ a trajectory-based approach for stability analysis,
where contraction theory plays a key role.
\end{abstract}


\maketitle

\section{Introduction}
\subsubsection*{Motivation and literature review}
In modern control systems,
shared networks and digital
platforms are commonly used for implementing feedback laws.
To effectively apply control theory in these systems,
it is imperative to consider resource constraints, 
including bandwidth and energy.
Two crucial elements in such resource constraints are quantization and transmission frequency.
Furthermore, it is important to be able to deal with nonlinear dynamics.
In this paper, we address these three aspects in a unified and systematic way.

Quantized control is motivated by numerous applications with limited communication capacity. 
The necessity for quantization also arises   due to 
physical  constraints on sensors and actuators.
It has been shown in \cite{Elia2001} that the coarsest quantization that quadratically 
stabilizes a linear discrete-time system with a single input is logarithmic.
In \cite{Fu2005},
an alternative proof for this result has been provided based on the sector bounded method.
An adaptive control framework for continuous-time 
nonlinear uncertain systems with input
logarithmic quantizers has been developed in \cite{Hayakawa2009}.
Logarithmic quantizers have been applied to various classes of systems such as 
Markov jump time-delay systems \cite{Shen2019} and 
parabolic partial differential equations \cite{Selivanov2016PDE, Kang2023}.
On the other hand, zooming quantizers, i.e., finite-level  quantizers with 
adjustable zoom parameters have been developed for
global asymptotic stabilization of linear systems in \cite{Brockett2000,Liberzon2003}.
This technique has been extended to  
nonlinear systems \cite{Liberzon2003Automatica,Liberzon2005}, switched linear systems
\cite{Liberzon2014, Wakaiki2017TAC}, and so on.
See also the overview \cite{Nair2007} for quantized control.

The implementation of controllers on digital platforms requires 
time-sampling. While periodic sampling is straightforward to apply and is commonly used in control applications, it may lead to resource overconsumption.
The need to use resources efficiently in networked control systems
has motivated the study of event-based aperiodic sampling, which
comprises  two major sub-branches called event-triggered control \cite{Arzen1999,Astrom2002, 
	Tabuada2007,Heemels2008} and
self-triggered control \cite{Velasco2003,Wang2009, Anta2010,Mazo2010}.
In event-triggered control, sensors monitor the measurement data
of the plant continuously or periodically, and send the data 
to the controller only when the triggering condition is satisfied.
To reduce the effort of monitoring  the measurement data,
self-triggering mechanisms (STMs) determine the next sampling time from the  measurement data at the
present 
and past sampling times.

For self-triggered control of nonlinear systems,
various techniques have been introduced, e.g.,
approximation of isochronous manifolds \cite{Anta2012,Delimpaltadakis2021TAC,
	Delimpaltadakis2021},
polynomial approximation of Lyapunov functions \cite{Benedetto2013}, and
reduction of conservativeness by disturbance observers \cite{Tiberi2013}.
Approaches based on the small gain theorem \cite{Tolic2012, Liu_Jiang2015} and 
control Lyapunov functions \cite{Proskurnikov2020} have also been discussed.
In \cite{Theodosis2019},
an updating threshold
strategy  has been developed.  
A dynamic STM has been proposed in \cite{Hertneck2021},
by combining a hybrid Lyapunov function and a dynamic variable
encoding the past system behavior.
Moreover,
self-triggered impulsive control for  nonlinear time-delay systems has been 
applied to a dose-regimen
design in \cite{Aghaeeyan2021}.

Many methods for quantized event-triggered control have been developed;
see, e.g., \cite{Garcia2013, Li2016ETC, Tanwani2016ETC ,Du2017,Liu2018ETC, Abdelrahim2019,Wang2021EJC,Zhao2021_IJRNC, Mazenc2022,
	Fu2022,Borri2022,Liu2023,Borri2024
} and the references therein. 
However, there are relatively small number of approaches available for
quantized self-triggered control.
In \cite{Song2022}, a self-triggered sliding mode control method has been
proposed for permanent magnet synchronous motors with
logarithmic quantization errors.
A zooming quantizer
and an STM have been jointly designed for continuous-time linear 
systems in \cite{Zhou2018STC} and for
discrete-time Takagi-Sugeno fuzzy systems in \cite{Wei2023}. In these studies
\cite{Song2022,Zhou2018STC,Wei2023},
the controller uses the quantized data, whereas the STM relies on 
the unquantized data.
In \cite{Gleizer2020},  input-to-state stability with respect to
bounded disturbances and noise has been achieved by a self-triggered strategy
for continuous-time linear systems.
The co-design of a zooming quantizer
and an STM using quantized measurements has been explored for 
continuous-time input-affine nonlinear systems in \cite{Lou2021}
and for
discrete-time linear systems in \cite{Wakaiki2023, Liu2023Automatica}.

\subsubsection*{Problem description}
This paper addresses the problem of self-triggered stabilization for
nonlinear systems with quantized state measurements.
The state is measured at sampling times determined 
by the STM, and only the quantized state is available to both
the controller and the STM.
Throughout this paper, the measurement error 
refers to the difference between 
the unquantized value of the current state and 
the quantized value of the last sampled state.
The plant input is kept constant between two consecutive sampling times
in a zero-order-hold (ZOH) fashion. 
Fig.~\ref{fig:closed_loop} illustrates the closed-loop system we consider.
The objective of this study is to design strategies for 
quantized  self-triggered control
to ensure closed-loop stability while avoiding Zeno behavior.

We focus particularly 
on contracting systems, i.e., systems
whose flow is an infinitesimally contraction mapping.
Contracting systems have highly-ordered asymptotic property and robustness against disturbances and noise.
Contraction theory was applied to control problems in 
\cite{Lohmiller1998} and has been
the subject of extensive research in control theory since then; see the  overview \cite{Tsukamoto2021}, the book \cite{Bullo2024}, and the references therein.
We employ
the contraction framework
for non-Euclidean norms established in \cite{Davydov2022,Davydov2025}.

We assume that 
the following basic conditions are satisfied on a certain region containing the origin:
\begin{enumerate}
	\renewcommand{\labelenumi}{(\roman{enumi})}
	\item 
	The ideal closed-loop system without quantization or sampling
	is contracting.
	\item
	The growth rate of the open-loop system is bounded by a known constant.
	\item
	The function representing the closed-loop 
	dynamics 
	has a local Lipchitz property with respect to the measurement error.
\end{enumerate}
By conditions (i) and (iii), 
the closed-loop system is input-to-state stable
with respect to the measurement error, which
guarantees that the state converges to the origin
under quantization and self-triggered sampling.
On the other hand,
condition (ii) is used to estimate the magnitude of the
measurement error for the computation of the next sampling time.

\begin{figure}[tb]
	\centering
	\includegraphics[width = 6cm]{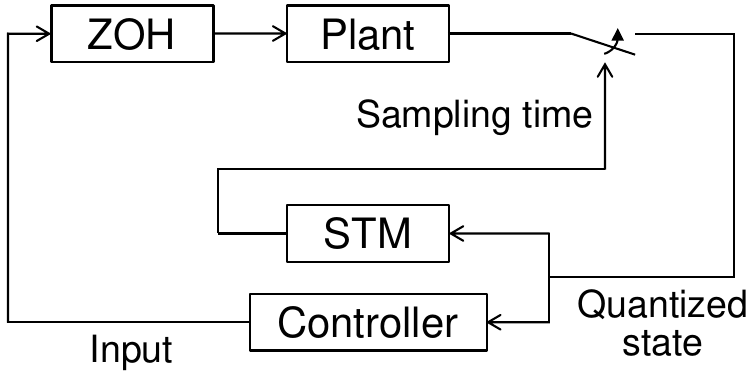}
	\caption{Closed-loop system.}
	\label{fig:closed_loop}
\end{figure}

\subsubsection*{Comparisons and contributions}
Event-triggered 
mechanisms using quantized data have been proposed, e.g.,
in \cite{Garcia2013,Liu2018ETC, Borri2022,Borri2024, Liu2023}.
In these mechanisms, triggering conditions 
are formulated based on the current quantized data being considered for transmission.
However, such data cannot be used in triggering conditions 
for self-triggered control.
To address
the lack of information, the proposed STMs predict the future state
trajectory from the quantized data. Triggering conditions
in the STMs are designed to account for prediction errors
induced by quantization.

Model uncertainties and disturbances have been addressed 
in the existing studies
on self-triggered control for nonlinear systems  \cite{Delimpaltadakis2021, Benedetto2013, Tiberi2013, Tolic2012, Liu_Jiang2015, Hertneck2021}.
On the other hand,
this paper deals with quantization errors within the 
framework 
of self-triggered control.
Quantization errors introduce initial state uncertainty 
in predicting future state trajectories, which
requires a different approach from that used for
model uncertainty and disturbances.
Therefore, the techniques proposed in the aforementioned existing studies 
cannot be directly applied to our problem setting.
Another important distinction is that quantization errors can be
regulated by setting quantizer parameters appropriately.
This motivates the development of interdependent 
design methods for quantization and 
self-triggered sampling.

The proposed design methods of
STMs using quantized measurements are
natural extensions of the discrete-time linear case 
studied in \cite{Wakaiki2023}.
In the previous study \cite{Lou2021} 
for continuous-time  input-affine nonlinear systems,
a zooming quantizer whose quantization center 
is the estimate of the plant state has been applied.
In contrast,
the quantization center of 
the zooming quantizer we use is the origin.
This eliminates the need for state estimation
at the sensor and hence 
reduces the computational burden 
on the sensor.

First, we study
self-triggered control with logarithmic quantization.
The main objective of the STM is to check whether
the measurement error exceeds a threshold like standard triggering mechanisms.
To this end, the STM also monitors whether the predicted state leaves 
the region where an upper bound on the
growth rate of the open-loop system is known.
To exploit local properties of the system behavior,
the norm of the initial state is assumed to be upper-bounded by a given constant.
Then we show that 
when the quantization density and 
the threshold parameter of self-triggered sampling
satisfy suitable conditions,
the closed-loop system has the following two properties: 
(i) inter-sampling times are uniformly lower-bounded by a strictly positive constant.
(ii) the state norm decreases monotonically and exponentially.

Our second contribution is to propose
a joint design method of a zooming quantizer and an STM.
The STM checks the conditions on the measurement error and the predicted
state as in the logarithmic quantization case. The major difference is that
the zooming quantizer and the STM are updated in
relation to each other.
The sampling times computed by 
the STM are aperiodic but belong to the discrete set $\{
p h : p = 1,2,3,\dots
\}$ for some period $h\in\mathbb{R}_{>0}$
as in
periodic event-triggered control systems studied, e.g., in 
\cite{Proskurnikov2020,
	Zhao2021_IJRNC, Fu2022,
	Heemels2013, Wang2019}.
Therefore, 
Zeno behavior does not occur.
However, the period $h$ for
the STM has to be small enough to guarantee the state convergence.
We provide a sufficient condition for stabilization, which
is described by inequalities with respect to the parameters for 
the range and step of quantization and for
the threshold and period of self-triggered sampling.

\subsubsection*{Paper organization}
The rest of this article is as follows. 
In Section~\ref{sec:basic_assump},
basic assumptions on the nonlinear system we consider are discussed. Then we
present preliminary results on the open-loop dynamics in Section~\ref{sec:open_loop_dynamics}.
In Section~\ref{sec:log_case},
we propose a self-triggered stabilization scheme incorporating logarithmic quantization.
Section~\ref{sec:zoom_case} is devoted to the problem of
jointly designing a zooming quantizer and an STM.
In Section~\ref{sec:lure}, 
the assumptions of the proposed methods
are examined for Lur'e systems.
We give a numerical example in
Section~\ref{sec:simulation}. Finally,
Section~\ref{sec:conclusion} concludes this article.

\subsubsection*{Notation}
If a function $f \colon \mathbb{R}^n \to \mathbb{R}^n$ is 
differentiable at $x \in \mathbb{R}^n$, then
we denote the Jacobian matrix of $f$ at $x$ by 
$Df(x)$.
For a matrix $A \in \mathbb{R}^{n \times n}$ with 
$(i,j)$-entry $A_{ij}$, 
the $(i,j)$-entry  of
the Metzler majorant $\lceil A\rceil_{\mathrm{Mzr}} \in \mathbb{R}^{n \times n}$ 
of $A$ is defined by
\[
(\lceil A\rceil_{\mathrm{Mzr}})_{ij} \coloneqq 
\begin{cases}
	A_{ij} & \text{if $i=j$} \\
	|A_{ij}| & \text{if $i\not=j$}. 
\end{cases}
\] 
The elementwise inverse of a vector $\theta =
[	\theta_{1}~
\cdots ~
\theta_{n}
]^{\top} \in \mathbb{R}^n_{>0}
$ is denoted by $\theta^{-1}$, i.e.,
$
\theta^{-1} \coloneqq 
[	\theta_1^{-1}~
\cdots ~
\theta_n^{-1}
]^{\top} 
$.
For a vector $\theta \in \mathbb{R}^n$, let  $[\theta] \in \mathbb{R}^{n \times n}$
denote
the diagonal matrix whose $i$-th diagonal element is equal to 
the $i$-th element of $\theta$.
When $\theta \in \mathbb{R}^n_{>0}$ and $r \in [1,\infty]$, 
we define the $\theta$-diagonally-weighted $r$-norm on $\mathbb{R}^n$ by
\[
\|x\|_{r,[\theta] } \coloneqq \big\|[\theta]  x \big\|_{r},\quad x \in \mathbb{R}^n.
\]
Let $\|\cdot\|_*$ be a norm on $\mathbb{R}^n$ and
its corresponding induced norm on $\mathbb{R}^{n \times n}$.
The logarithmic norm of a matrix $A \in \mathbb{R}^{n \times n}$
with respect to the norm $\|\cdot\|_*$ is defined by
\[
\mu_{*}(A) \coloneqq \lim_{\varepsilon  \to 0+} \frac{\|I + \varepsilon A\|_* - 1}{\varepsilon }.
\]
Let $\mathbf{B}_*(R)$ be the open ball of radius $R \in \mathbb{R}_{>0} \cup \{\infty \}$ centered at $0$
with respect to the norm $\|\cdot\|_*$, i.e.,
\[
\mathbf{B}_*(R) \coloneqq 
\begin{cases}
	\{
	x \in \mathbb{R}^n :\|x\|_* < R
	\} & \text{if $0< R < \infty$}\\
	\mathbb{R}^n & \text{if $R = \infty$},
\end{cases}
\]
and 
let $\overline{\mathbf{B}}_*(R)$
denote the closure of $\mathbf{B}_*(R)$.

\section{Basic assumptions and system properties}
\label{sec:basic_assump}
In this section, we first discuss basic assumptions on the closed-loop system and the open-loop system. Then we present useful results 
obtained under those assumptions.
These results give
growth and decay properties of the systems, which will be used for
the design of STMs and the stability analysis of quantized self-triggered control systems.

\subsection{Assumptions on nonlinear systems}
\label{sec:nonlinear_assump}
Throughout this paper,
let  $f \colon \mathbb{R}^n \times \mathbb{R}^m \to \mathbb{R}^n$ and $g \colon \mathbb{R}^n \to \mathbb{R}^m$ be continuous functions
satisfying
$f(0,g(0)) = 0$.
Consider the ordinary differential equation (ODE)
\[
\dot x(t) = f\big(x(t),u(t)\big),\quad t \in \mathbb{R}_{\geq 0};
\qquad x(0) = x_0 \in \mathbb{R}^n,
\]
where $x(t) \in \mathbb{R}^n$ and $u(t) \in \mathbb{R}^m$
are the state and the input of the nonlinear system at time $t$, respectively.
In the ideal case without quantization or self-triggered sampling,
the input $u$ is generated by the following controller:
\[
u(t) = g\big(x(t)\big),\quad t \in \mathbb{R}_{\geq 0}.
\]

For $
x,e \in \mathbb{R}^n$, we define
\[
F (x,e) \coloneqq f\big(x,g(x+e)\big) \quad \text{and} \quad 
F_{0} (x) \coloneqq F(x,0).
\]
The ideal closed-loop dynamics can be written as the ODE $\dot x = F_0(x)$.
First, we make the following assumption on $F_0$ to ensure 
the contractivity of 
the ideal closed-loop system.
\begin{assumption}
	\label{assump:closed}
	There exist
	a norm $\|\cdot\|_{\cl}$ on $\mathbb{R}^n$ and
	constants $c\in \mathbb{R}_{>0}$ and $R_1 \in \mathbb{R}_{>0} \cup \{\infty \}$ such that
	\begin{enumerate}
		\renewcommand{\labelenumi}{(\roman{enumi})}
		\item  $F_0$ is locally Lipschitz continuous on $\bcl(R_1)$; and \vspace{2pt}
		\item $\mu_{\cl}( DF_0(x)) \leq -c$ for almost all $x \in \bcl(R_1)$.
	\end{enumerate}
\end{assumption}

Note that 
if $F_0$ is locally Lipschitz continuous on $\bcl(R_1)$,
then $F_0$ is differentiable almost everywhere on $\bcl(R_1)$ by
Rademacher’s theorem; see, e.g., \cite[Theorem~6.15]{Heinonen2001}.

For a fixed $q \in \mathbb{R}^n$,
we set
\[
f_q(x) \coloneqq f\big(x,g(q)\big),\quad x \in \mathbb{R}^n.
\]
When the input $u$ is a constant function such that $u(t) \equiv g(q)$ for some 
$q \in \mathbb{R}^n$, we can write 
the open-loop dynamics as
the ODE $\dot x = f_q(x)$.
Next, we assume that $f_q$ satisfies the following properties, which give
an upper bound on the growth rate
of the open-loop system.
\begin{assumption}
	\label{assump:open}
	There exist
	a norm $\|\cdot\|_{\op}$ on $\mathbb{R}^n$ 
	and constants $d_1\in \mathbb{R}_{\geq 0}$,
	$d_2 \in \mathbb{R}_{>0}$, and $R_2 \in \mathbb{R}_{>0} \cup \{\infty \}$ 
	such that
	\begin{enumerate}
		\renewcommand{\labelenumi}{(\roman{enumi})}
		\item
		$f_q$ is 
		continuously differentiable on 
		$\bop(R_2)$ for all $q \in \bop(R_2)$; \vspace{2pt}
		\item $\mu_{\op}( Df_q(x))  \leq d_1 $ for all $x,q \in \bop(R_2)$; and \vspace{2pt}
		\item 
		$
		\|f_q(q) \|_{\op} \leq d_2 \|q\|_{\op}
		$
		for all $q \in \bop(R_2)$.
	\end{enumerate}
\end{assumption}

Note that the norm $\|\cdot\|_{\op}$ in Assumption~\ref{assump:open}
might differ from the norm $\|\cdot\|_{\cl}$ in 
Assumption~\ref{assump:closed}.
The norm $\|\cdot\|_{\cl}$ is applied for the stability analysis of quantized self-triggered
control systems, whereas 
the proposed STMs estimate the magnitude of the measurement error
by using  the norm $\|\cdot\|_{\op}$.

For the norms $\|\cdot\|_{\cl}$ and $\|\cdot\|_{\op}$ in 
Assumptions~\ref{assump:closed} and \ref{assump:open},
let a constant $\Gamma \in \mathbb{R}_{>0}$ satisfy 
\[
\|x\|_{\op} \leq \Gamma \|x\|_{\cl}
\]
for all $x \in \mathbb{R}^n$.
Such a constant exists, since any two norms on $\mathbb{R}^n$
are equivalent.
Put 
\[
R \coloneqq \min \left\{R_1,\,\frac{R_2}{\Gamma} \right\}.
\]
Then 
\[
\bcl(R) \subseteq \bcl(R_1) \cap \bop(R_2).
\]

Finally, we assume that the function 
$F(x,e)$ has a certain local Lipschitz property in 
the second argument $e$. 
This Lipschitz property will be used to obtain an upper bound on
the effect of the measurement error.
\begin{assumption}
	\label{assump:control_part}
	There exist constants $\alpha,\sigma_0  \in \mathbb{R}_{>0}$ such that 
	\begin{equation}
		\label{eq:F_e_Lip}
		\|F(x,e) - F(x,0)\|_{\cl}\leq \alpha \|e\|_{\op}
	\end{equation}
	for all $x\in \bcl(R)$ and $e \in \bop(\sigma_0 R)$.
\end{assumption}

We will 
consider state trajectories staring from
$\bcl(R)$ or a smaller ball 
in the stability analysis of quantized self-triggered control systems.
This condition on initial states is generally restrictive, and
if the initial state is outside of the ball, then
we have to drive the state into the ball by existing control methods without self-triggered sampling before applying the proposed methods.
We also make the following comment on
the selection of the norms $\|\cdot\|_{\cl}$ and $\|\cdot\|_{\op}$.
\begin{remark}[Selection of norms]
	When designing a quantizer and an STM,
	it is important to select the norms $\|\cdot\|_{\cl}$ and $\|\cdot\|_{\op}$
	so that the constants $-c$  and $d_1$ are small.
	For nonlinear systems in Lur'e form, design methods of
	such norms have been presented
	with respect to 
	weighted $2$-norms  in \cite[Section~3.7]{Bullo2024} 
	and with respect to 
	diagonally-weighted $1$-norms and $\infty$-norms in \cite[Section~VI.D]{Davydov2025}.
	We will explain
	the case of diagonally-weighted  $\infty$-norms
	in Section~\ref{sec:lure}.
	For general nonlinear systems, a typical method is to linearize
	the nonlinear system. Then,  design methods
	of norms in the linear case, which are
	summarized in
	\cite[Section~2.7]{Bullo2024}, can be applied.
	\hspace*{\fill} $\triangle$ 
\end{remark}

\subsection{Growth and decay properties}
Using the contraction theory,
we see that the closed-loop system has an
input-to-state stability property 
with respect to the measurement error
under Assumptions~\ref{assump:closed} and \ref{assump:control_part}.
Assumption~\ref{assump:open} implies
an incremental property  of the open-loop system, specifically,
an upper bound on the growth rate of
the difference between two trajectories.
In fact,
the conditions on the logarithmic norms  in 
Assumptions~\ref{assump:closed} and \ref{assump:open}
can be converted to
those on the minimal one-sided Lipschitz constants; see
\cite[Theorem~17]{Davydov2025}. Therefore,
by slightly modifying
\cite[Theorems~31 and 37]{Davydov2022},
we can obtain the following two results.
\begin{theorem}
	\label{thm:basic_prop_of_contraction3}
	Suppose that Assumptions~\ref{assump:closed} and \ref{assump:control_part} 
	hold. Let $\tau \in \mathbb{R}_{>0}$, and let $e\colon
	[0,\tau) \to \bop(\sigma_0 R)$ be continuous.
	If the ODE
	$\dot x = F(x,e)$ with $x(0) = x_0\in \bcl(R)$ has a 
	solution $x$ on $[0,\tau]$ such that $x(t) \in \bcl(R)$ for all $t \in [0,\tau)$,
	then
	\[
	\|x(t)\|_{\cl} \leq e^{-ct} \|x_0\|_{\cl} + 
	\frac{\alpha (1-e^{-ct})}{c} \sup_{0\leq s < \tau }
	\|e(s)\|_{\op}
	\]
	for all $t \in [0,\tau]$.
\end{theorem}
\begin{theorem}
	\label{thm:basic_prop_of_contraction2}
	Suppose that Assumption~\ref{assump:open} holds.
	Let $\tau \in \mathbb{R}_{>0}$ and $q \in \bop(R_2)$.
	Assume that for each $i=1,2$, 
	the ODE $\dot x_i = f_q(x_i)$
	with $x_i(0) = x_{i,0} \in \bop(R_2)$ 
	has a solution $x_i$ on $[0,\tau)$
	such that
	$x_i(t) \in \bop(R_2)$ for all  $t \in [0,\tau)$. Then
	\[
	\|x_1(t)-x_2(t)\|_{\op} \leq e^{d_1t} \|x_{1,0} - x_{2,0}\|_{\op}
	\]
	for all  $t \in [0,\tau)$.
\end{theorem}

\section{Preliminary results on open-loop dynamics}
\label{sec:open_loop_dynamics}
Let $q \in \bop(R_2)$, and consider the ODE
\begin{equation}
	\label{eq:x_q_ODE}
	\dot x_q(t) = f_q\big(x_q(t)\big) = f\big(x_q(t),g(q) \big),\quad x_q(0) =q.
\end{equation}
We will use the ODE \eqref{eq:x_q_ODE} in order for 
the proposed STMs to predict the state trajectory on sampling intervals.
In this setting, the initial value $q$ is the quantized state at a sampling time.
For the design of STMs, here we present three properties of the ODE \eqref{eq:x_q_ODE}.

The first lemma is used to obtain an upper bound of the measurement error.
It can be proved easily but is the basis for the 
design of STMs.
\begin{lemma}
	\label{lem:x_q_diff}
	Suppose that Assumption~\ref{assump:open} holds.
	Let $\tau \in \mathbb{R}_{>0}$ and 
	$q,x_0 \in \bop(R_2)$.
	Assume that the ODEs
	\eqref{eq:x_q_ODE}
	and 
	\[
	\dot x(t) = f\big(x(t),g(q)\big),\quad x(0) = x_0
	\] 
	have solutions
	$x_q$ and $x$ on $[0,\tau)$, respectively.
	If $x_q(t),x(t)  \in \bop(R_2)$ for all $t \in [0,\tau )$, then
	\begin{equation}
		\label{eq:x_q_diff}
		\| x(t) - q \|_{\op} \leq 
		e^{d_1 t} \|x_0 - q\|_{\op}  + 
		\|x_{q}(t) - q \|_{\op}
	\end{equation}
	for all $t \in [0,\tau )$.
\end{lemma}
\begin{proof}
	By the triangle inequality, we have 
	\begin{align*}
		\| x(t) - q\|_{\op}  \leq 
		\| x(t) -  x_{q}(t) \|_{\op} + 
		\| x_{q}(t) - q \|_{\op}.
	\end{align*}
	Applying Theorem~\ref{thm:basic_prop_of_contraction2} 
	to $\| x(t) -  x_{q}(t) \|_{\op}$,
	we immediately obtain the desired conclusion.
\end{proof}

Next, we derive an upper bound on
the difference between the solution $x_q(t)$ of the ODE \eqref{eq:x_q_ODE} and the initial value $q$. 
\begin{lemma}
	\label{lem:is_diff}
	Suppose that Assumption~\ref{assump:open} holds.
	Let $\tau \in \mathbb{R}_{>0}$  and 
	$q \in \bop(R_2)$.
	Assume that 
	the ODE \eqref{eq:x_q_ODE} has a solution $x_q$
	on $[0,\tau)$ such that
	$x_q(t) \in \bop(R_2)$ for all $t \in [0,\tau)$. Then
	\begin{equation}
		\label{eq:xq_q_bound}
		\|x_q(t) - q 
		\|_{\op} \leq \nu(t) \|q\|_{\op} 
	\end{equation}
	for all $t \in [0,\tau)$, where $\nu(t)$ is defined by
	\begin{equation}
		\label{eq:nu_def}
		\nu(t) \coloneqq 
		d_2 \int_0^t e^{d_1s} ds
		=
		\begin{cases}
			\dfrac{d_2(e^{d_1 t} - 1)}{d_1} & \text{if $d_1 \not=0$} \vspace{5pt}\\
			d_2 t & \text{if $d_1 =0$}
		\end{cases}
	\end{equation}
	for $t \in \mathbb{R}_{\geq 0}$.
\end{lemma}
\begin{proof}
	By modifying \cite[Theorem~3.9.(ii)]{Bullo2024} slightly,
	we have
	\[
	\big\|f_q\big(x_q(s)\big) \big\|_{\op} 
	\leq e^{d_1s} \big\|f_q\big(x_q(0) \big) \big\|_{\op}
	\]
	for all $s \in [0,\tau)$.
	Since $x_q(0) = q$, it follows from Assumption~\ref{assump:open} that
	\[
	\big\|f_q\big(x_q(s)\big)\big\|_{\op} \leq 
	e^{d_1s} \|f_q(q)\|_{\op} \leq d_2e^{d_1s} \|q\|_{\op}
	\]
	for all $s \in [0,\tau)$.
	Thus,
	\begin{align*}
		\|x_q(t) - q 
		\|_{\op} 
		\leq 
		\int_{0}^t \left\|f_q\big(x_q(s) \big) \right\|_{\op} ds
		\leq 
		d_2 \int_0^t e^{d_1s} ds\|q\|_{\op} 
	\end{align*}
	is obtained for all $t \in [0,\tau)$.
\end{proof}

Let $\lambda \in (0,1)$.
Finally, we give a time period on which
all trajectories starting in $\bop(\lambda R_2)$
are guaranteed to stay in $\bop(R_2)$.
Denote by $\widetilde\tau_{\min} = \widetilde\tau_{\min}(\lambda)$ the solution of the equation
$
\lambda \left(
1+ \nu(t)
\right) = 1,
$
i.e.,
\begin{align}
	\label{eq:tilde_tau_def}
	\widetilde\tau_{\min}(\lambda) \coloneqq
	\begin{cases}
		\dfrac{1}{d_1} \log\left(
		1 + \dfrac{d_1(1-\lambda)}{d_2 \lambda} 
		\right) & \text{if $d_1 \not=0$} \vspace{5pt}\\
		\dfrac{1-\lambda}{d_2 \lambda} & \text{if $d_1 = 0$}.
	\end{cases}
\end{align}
\begin{lemma}
	\label{lem:open_loop_bound}
	Suppose that Assumption~\ref{assump:open} holds,
	and let $\lambda \in (0,1)$ and $q \in \bop(\lambda R_2)$.
	Then
	there exists a unique solution $x_q$
	of the ODE \eqref{eq:x_q_ODE} on 
	$[0, \widetilde \tau_{\min}]$.
	Moreover, $x_q$ satisfies
	\begin{equation}
		\label{eq:x_q_bound}
		x_q(t) \in \bop(R_2)
	\end{equation}
	for all $t \in
	[0, \widetilde \tau_{\min}]$.
\end{lemma}
\begin{proof}
	It is enough to show that the following two statements hold for
	fixed $\lambda \in (0,1)$ and 
	$q \in \bop(\lambda R_2)$:
	\begin{enumerate}
		\renewcommand{\labelenumi}{(\roman{enumi})}
		\item The ODE \eqref{eq:x_q_ODE} 
		has a unique solution  $x_q$
		on $
		[0, \widetilde \tau_{\min}]
		$.
		\item The solution $x_q$ satisfies
		\begin{equation}
			\label{eq:xq_bound_proof}
			\|x_q(t)\|_{\op} \leq
			\big(
			1+ \nu(t)
			\big) \|q\|_{\op}\quad \text{for all
				$t \in
				[0, \widetilde \tau_{\min} ]$.}
		\end{equation}
	\end{enumerate}
	Indeed, since 
	$
	\lambda (1+\nu(t) ) \leq 1
	$ 
	for all $t \in [0,\widetilde \tau_{\min}]$,
	the inequality \eqref{eq:xq_bound_proof} yields
	\begin{equation}
		\label{eq:xq_bound}
		x_q(t) \in \overline{\mathbf{B}}_{\op}( \|q\|_{\op} / \lambda  )
	\end{equation}
	for all 
	$t \in
	[0, \widetilde \tau_{\min} ]$.
	From $\|q\|_{\op} < \lambda R_2$, we obtain
	\begin{equation}
		\label{eq:bop_q_lam}
		\overline{\mathbf{B}}_{\op}( \|q\|_{\op} / \lambda  )  \subset \bop(R_2).
	\end{equation}
	Combining 
	\eqref{eq:xq_bound} and \eqref{eq:bop_q_lam},
	we conclude 
	that \eqref{eq:x_q_bound} holds.
	
	We prove that statements (i) and (ii) are true, by obtaining a contradiction.
	Assume that 
	\begin{itemize}
		\item
		the  solution
		of the ODE \eqref{eq:x_q_ODE} either does not exist
		or is not unique on
		$
		[0, \widetilde \tau_{\min}]
		$; or that
		\item
		there exists a unique solution $x_q$
		of the ODE \eqref{eq:x_q_ODE} on
		$
		[0, \widetilde \tau_{\min}]$, but
		$x_q$ does not satisfy the inequality \eqref{eq:xq_bound_proof}.
	\end{itemize}
	In both cases, we can deduce that 
	there exists $s_0 \in (0,\widetilde \tau_{\min}]$ such that 
	the ODE \eqref{eq:x_q_ODE}  has a unique  solution $x_q$
	on
	$
	[0, s_0]$ satisfying
	\[
	\|x_q(s_0)\|_{\op} > \big(
	1+\nu(s_0)
	\big) \|q\|_{\op}.
	\]
	Define
	\[
	s_1 \coloneqq 
	\inf\left\{
	t \in \mathbb{R}_{> 0} :
	\|x_q(t)\|_{\op} >
	\big(
	1+ \nu(t)
	\big) \|q\|_{\op}
	\right\} \in [0,s_0].
	\]
	Since $x_q$ and $\nu$ are continuous, 
	it follows that 
	$
	s_1 < s_0 \leq \widetilde  \tau_{\min},
	$
	and then,
	for all $t \in [0, s_1]$, 
	\[
	\|x_q(t)\|_{\op} \leq 
	\big(
	1+ \nu(t)
	\big) \|q\|_{\op}< \lambda \big(
	1+ \nu(\widetilde  \tau_{\min})
	\big)R_2 =  R_2.
	\]
	Hence
	there exists $\delta \in (0,s_0-s_1)$ such that
	$
	\|x_q(t)\|_{\op} < R_2
	$
	for all $t \in [0,s_1+ \delta )$. Using 
	Lemma~\ref{lem:is_diff}, we obtain
	\[
	\|x_q(t)\|_{\op} \leq 
	\|q\|_{\op} + \|x_q(t) - q\|_{\op} \leq
	\big(
	1+\nu(t)
	\big) \|q\|_{\op}
	\]
	for all $t \in [0,s_1+\delta)$.
	This contradicts the definition of $s_1$.
\end{proof}

\section{Self-triggered stabilization under logarithmic quantization}
\label{sec:log_case}
In this section, we study the problem of self-triggered stabilization for
contracting systems under logarithmic quantization.
First, we describe the logarithmic quantizer used in this paper and 
present some basic properties.
Next, we propose an STM that computes inter-sampling times from
the system model and the quantized state.
Finally, we show that the state converges to the origin
without Zeno behavior if the parameters of
quantization and self-triggered sampling
are chosen suitably.

\subsection{System model}
\subsubsection{Logarithmic quantizer}
Let $\rho \in (0,1)$ and $\chi_{i,0} \in \mathbb{R}_{>0}$ for $i=1,\dots,n$, where $n$ is the 
dimension of the state space.
Set $\chi_{i,j} \coloneqq \rho^j \chi_{i,0}$ for $i=1,\dots,n$ and $j \in \mathbb{Z}$.
For $i=1,\dots,n$,
define the logarithmic quantization function 
$Q_{\log,i} \colon \mathbb{R} \to \mathbb{R}$ by
\begin{align*}
	Q_{\log,i}(z) \coloneqq
	\begin{cases}
		\dfrac{\chi_{i,j} + \chi_{i,j+1}}{2} & \text{if $\chi_{i,j+1} \leq z < \chi_{i,j}$} \vspace{3pt}\\
		0 & \text{if $z = 0$} \\
		-Q_{\log,i}(-z) & \text{if $z < 0$}.
	\end{cases}
\end{align*}
Using 
the scalar function
$Q_{\log,i}$, we also define the vector function 
$Q_{\log}\colon \mathbb{R}^n \to \mathbb{R}^n$ by
\begin{align*}
	Q_{\log}
	\left(
	\begin{bmatrix}
		z_1 \\ \vdots \\ z_n
	\end{bmatrix}
	\right) &\coloneqq
	\begin{bmatrix}
		Q_{\log,1}(z_1) \\ \vdots \\ 	Q_{\log,n}(z_n)
	\end{bmatrix}.
\end{align*}
The logarithmic quantizer by the function $Q_{\log}$ becomes 
coarser as the quantization density $\rho$ decreases.

To exploit  the property that the $i$-th element of $Q_{\log}(x)$ depends
only on the $i$-th element of $x$, 
we use diagonally-weighted norms in this section.
Take $\theta_{\cl},\theta_{\op} \in \mathbb{R}^n_{>0}$ and $r \in [1,\infty]$ 
arbitrarily. 
Since any two norms on $\mathbb{R}^n$ are equivalent,
it follows that
for the norms 
$\|\cdot\|_{\cl}$ and $\|\cdot\|_{\op}$
in
Assumptions~\ref{assump:closed} and \ref{assump:open}, 
there exist
constants $L_{\cl,1},L_{\cl,2},L_{\op,1},L_{\op,2} \in \mathbb{R}_{>0}$ such that
\begin{equation}
	\label{eq:cl_norm_equiv}
	L_{\cl,1} \|x \|_{\infty,[\theta_{\cl}]} \leq 
	\|x\|_{\cl} \leq L_{\cl,2} \|x \|_{\infty,[\theta_{\cl}]} 
\end{equation}
and
\begin{equation}
	\label{eq:op_norm_equiv}
	L_{\op,1} \|x \|_{r,[\theta_{\op}]} \leq 
	\|x\|_{\op} \leq L_{\op,2} \|x \|_{r,[\theta_{\op}]} 
\end{equation}
for all $x \in \mathbb{R}^n$. 
Notice that 
only $\infty$-norms are chosen for $\|\cdot\|_{\cl}$, because
a property of $\infty$-norms will be used
to derive 
a condition on the parameter $\chi_{i,0}$
of $Q_{\log, i}$ 
in Assumption~\ref{assump:initial_cond}
and Lemma~\ref{lem:Qlog_prop} below.
Define 
the constants $L_{\cl},L_{\op} \geq 1$ by
\begin{equation}
	\label{eq:Lcl_Lop_def}
	L_{\cl} \coloneqq \frac{L_{\cl,2}}{L_{\cl,1}}
	\quad \text{and} \quad 
	L_{\op} \coloneqq \frac{L_{\op,2}}{L_{\op,1}}.
\end{equation}
If $\|\cdot \|_{\cl} =  \|\cdot \|_{\infty,[\theta_{\cl}]} $ and 
$\|\cdot \|_{\op} =  \|\cdot \|_{r,[\theta_{\op}]} $, then 
$L_{\cl} = L_{\op} = 1$.
\begin{assumption}
	\label{assump:initial_cond}
	The parameter $\chi_{i,0}$ of the quantization function $Q_{\log,i}$
	is chosen as $\chi_{i,0} = \chi_0 / \theta_{\cl,i}$ for each $i=1,\dots,n$, 
	where
	$[\theta_{\cl,1}~~\theta_{\cl,2}~~\cdots~~\theta_{\cl,n} ] \coloneqq \theta_{\cl}^{\top}$
	and $\chi_0 \in \mathbb{R}_{>0}$ satisfies
	\[
	\frac{R}{L_{\cl,2}} \leq \chi_{0} < \frac{2R}{L_{\cl,2}(1+\rho)}
	\]
	for the constant $R \in \mathbb{R}_{>0}$ given in Section~\ref{sec:nonlinear_assump}.
\end{assumption}

The reason for the choice of the parameter 
$\chi_{0}$ in Assumption~\ref{assump:initial_cond}
is that
the quantized value of $x \in \bcl( R / L_{\cl})$ 
belongs to $\bcl(\lambda R)$ for some
$\lambda \in (0,1)$, which is proved in the next lemma.
\begin{lemma}
	\label{lem:Qlog_prop}
	Suppose that Assumptions~\ref{assump:closed}, \ref{assump:open}, and
	\ref{assump:initial_cond} hold.
	Define 
	\[
	\lambda_0 \coloneqq \frac{(1+\rho) L_{\cl,2} \chi_0}{2R}  < 1
	\]
	and let $\lambda \in (\lambda_0, 1)$.
	Then $Q_{\log}(x) \in \bcl(\lambda R)$ 
	for all $x \in \bcl( R / L_{\cl})$.
\end{lemma}
\begin{proof}
	Let	$i=1,\dots,n$ be given, and define $\widetilde R \coloneqq 
	R / L_{\cl,2}$.
	By Assumption~\ref{assump:initial_cond}, we obtain
	$\widetilde R / \theta_{\cl,i} \leq \chi_{i,0}$.
	If $z \in \mathbb{ R}$ satisfies $|z| < \widetilde R/\theta_{\cl,i}$, then
	\[
	\left|Q_{\log,i} \left(
	z
	\right) \right|
	\leq \frac{1+\rho }{2}\chi_{i,0} = \frac{\lambda_0 \widetilde R}{\theta_{\cl,i}} < 
	\frac{\lambda \widetilde R}{\theta_{\cl,i}}.
	\]
	Therefore, if $\|x\|_{\infty,[\theta]_{\cl}} < \widetilde R$, then
	$\|Q_{\log}(x)\|_{\infty,[\theta]_{\cl}} < \lambda \widetilde R$.
	This and the inequalities in \eqref{eq:cl_norm_equiv} show that 
	$Q_{\log}(x) \in \bcl(\lambda R)$ for all $x \in \bcl(R / L_{\cl})$.
\end{proof}

We present useful inequalities on quantized values and quantization errors.
Routine calculations show that
for all $\theta \in \mathbb{R}^n_{>0}$, $r \in [1,\infty]$, and $x \in \mathbb{R}^n$,
the quantization function $Q_{\log}$ satisfies
\[
\|Q_{\log}(x)\|_{r,[\theta]}  \leq \frac{1+\rho}{2\rho} \|x\|_{r,[\theta]}
\]
and
\[
\|Q_{\log}(x) - x\|_{r,[\theta]}  \leq \frac{1-\rho}{1+\rho} \|Q_{\log}(x)\|_{r,[\theta]}.
\]
Using \eqref{eq:cl_norm_equiv} and \eqref{eq:op_norm_equiv}, we obtain
\begin{equation}
	\label{eq:log_prop2}
	\|Q_{\log}(x)\|_{\cl}  \leq \frac{L_{\cl}(1+\rho)}{2\rho} \|x\|_{\cl}
\end{equation}
and
\begin{equation}
	\label{eq:log_prop1}
	\|Q_{\log}(x) - x\|_{\op}  \leq \frac{L_{\op}(1-\rho)}{1+\rho} \|Q_{\log}(x)\|_{\op}.
\end{equation}

\subsubsection{Self-triggered control system}
Let 
$\{t_k\}_{k \in \mathbb{Z}_{\geq 0} } 
$  be a strictly increasing sequence with $t_0 \coloneqq 0$.
Consider the following closed-loop system:
\begin{equation}
	\label{eq:closed_loop}
	\left\{
	\begin{aligned}
		\dot x(t) &= f\big(x(t),u(t)\big), && 
		t \in \mathbb{R}_{\geq 0};\qquad x(0) = x_0 \\
		u(t) &= g(q_k ),&&  t \in [t_k, t_{k+1}),~k \in \mathbb{Z}_{\geq 0} \\
		q_k &= Q_{\log}\big(x(t_k)\big),&&  k \in \mathbb{Z}_{\geq 0}.
	\end{aligned}
	\right.
\end{equation}
Inspired by 
the inequality \eqref{eq:x_q_diff} and 
the property \eqref{eq:log_prop1} of the logarithmic quantizer,
we define the function $\psi_{\log}$ by
\begin{equation}
	\label{eq:psi_log}
	\psi_{\log}(\tau,q) \coloneqq  
	\frac{L_{\op}(1-\rho)}{1+\rho} e^{d_1 \tau} \|q\|_{\op} +  
	\|x_{q}(\tau)- q \|_{\op}
\end{equation}
for $\tau \in \mathbb{R}_{\geq 0}$ and $q \in \bcl(R)$ such that there exists
a unique solution $x_q$ of the ODE \eqref{eq:x_q_ODE} on $[0, \tau]$.
Let $\tau_{\max} \in \mathbb{R}_{>0}$.
For the computation of the sampling times $\{t_k\}_{k \in \mathbb{Z}_{\geq 0} }$,
we propose 
the  STM given by 	
\begin{subequations}
	\label{eq:STM}
	\begin{empheq}[left = {\empheqlbrace \,}, right = {}]{align}
		t_{k+1} &\coloneqq t_k + \tau_k \\
		\tau_{k} &\coloneqq 
		\begin{cases}
			\widetilde \tau_k & \text{if 
				$\psi_{\log}(\tau, q_k) \leq \sigma \|q_k\|_{\cl}$
				for all $\tau \in (0, \widetilde \tau_k)$} \\
			\inf \{
			\tau \in  \mathbb{R}_{>0}  : 
			\psi_{\log}(\tau, q_k) > \sigma \|q_k\|_{\cl}
			\}
			&  \text{otherwise} 
		\end{cases} \label{eq:log_SMT2}\\
		\widetilde \tau_{k} &\coloneqq 
		\begin{cases}
			\tau_{\max} &
			\text{if $\|x_{q_k}(\tau)\|_{\op} < R_2$
				for all $\tau \in (0, \tau_{\max})$} \\
			\inf \{
			\tau \in \mathbb{R}_{>0} : 
			\|x_{q_k}(\tau)\|_{\op} \geq R_2
			\} \hspace{18pt}
			& \text{otherwise},
		\end{cases} \label{eq:log_SMT3}
	\end{empheq}
\end{subequations}
where
$x_{q_k}$ is the solution of the ODE \eqref{eq:x_q_ODE}  with $q=q_k$.
In the STM \eqref{eq:STM}, 
$\tau_{\max} $ is an upper bound of the
inter-sampling times. 
Indeed,
since $\tau_k \leq \widetilde \tau_k \leq \tau_{\max}$ by definition, we have 
$t_{k+1} - t_k \leq \tau_{\max}$.
\stepcounter{equation}

The objective of the SMT \eqref{eq:STM} is that
the measurement error, i.e.,
the difference between the state $x(t_k+\tau)$
and the quantized state $q_k$ at the sampling time $t= t_k$ is 
upper-bounded by
$\sigma \|q_k\|_{\cl}$.
We now briefly explain how the STM \eqref{eq:STM} works.
To compute the inter-sampling time $\tau_k$,
the STM~\eqref{eq:STM} 
predicts the future state trajectory from
$q_k$.
The predicted state is given by
the solution $x_{q_k}$ of the ODE \eqref{eq:x_q_ODE} 
with  $q=q_k$ and is 
computed  in practice by numerical methods for solving 
ODEs.
Then $x_{q_k}$ is used to evaluate the triggering conditions in \eqref{eq:log_SMT2}
and \eqref{eq:log_SMT3}.
Note that 
the true state $x(t_k+\tau)$ and 
the predicted state $x_{q_k}(\tau)$ have trajectories with different 
values at $\tau = 0$ 
due to the quantization error $x(t_k) - q_k$.

The roles of the triggering conditions in \eqref{eq:log_SMT2}
and \eqref{eq:log_SMT3} are as follows:
From \eqref{eq:log_SMT2}, we obtain 
\[
\psi_{\log}(\tau, q_k) \leq \sigma \|q_k\|_{\cl}
\]
for all $\tau \in [0,  \tau_k)$.
On the other hand, by
\eqref{eq:log_SMT3}, 
there exists a unique solution $x_{q_k}$ of the ODE \eqref{eq:x_q_ODE} with $q=q_k$
on $[0,\widetilde \tau_k]$, and
$x_{q_k}$ satisfies 
$x_{q_k} (\tau) \in \bop(R_2)$ for all $\tau \in [0,  \widetilde \tau_k)$.
Combining this and Lemma~\ref{lem:x_q_diff},
we will show that
\[
\|x(t_k+\tau) - q_k\|_{\op} \leq \psi_{\log}(\tau, q_k)
\]
for all $\tau \in [0,  \tau_k)$.
Note that 
this inequality is obtained by using the upper bound on the growth rate of the open-loop system on $\bop(R_2)$. Hence
the STM has to monitor whether the condition $x_{q_k}(\tau) \in \bop(R_2)$ is satisfied.
In the linear case \cite{Wakaiki2023, Liu2023}, growth bounds can be obtained on $\mathbb{R}^n$, and hence this monitoring is not required.

\subsection{Stability analysis}
\subsubsection{Main result in logarithmic quantization case}
Before stating the main result of this section,
we introduce two constants $\tau_{\min}$ and $\widetilde \tau_{\min}$
to describe a lower bound of the inter-sampling times.
We define 
$\tau_{\min} \in \mathbb{R}_{>0}$ by
the solution of the equation
\begin{equation}
	\label{eq:tau_min}
	\frac{L_{\op}(1-\rho)}{1+\rho} e^{d_1 t} + 
	\nu(t)
	= \frac{\sigma}{\Gamma},
\end{equation}
where the function $\nu$ is as in \eqref{eq:nu_def}.
Since $\nu(0) = 0$ and $\nu$ is strictly increasing,
there exists a unique solution of the equation 
\eqref{eq:tau_min} on $\mathbb{R}_{>0}$ if and only if 
\begin{equation}
	\label{eq:rho_sigma_ineq}
	\frac{\Gamma L_{\op}(1-\rho)}{1+\rho} <
	\sigma.
\end{equation}
When this inequality \eqref{eq:rho_sigma_ineq} holds,
$\tau_{\min} \in \mathbb{R}_{>0}$ is given by
\[
\tau_{\min} = 
\begin{cases}
	\dfrac{1}{d_1}
	\log\left(
	\dfrac{\sigma(1+\rho)d_1 + \Gamma(1+\rho)d_2}{\Gamma L_{\op}(1-\rho) d_1+\Gamma(1+\rho)d_2 }
	\right)  & \hspace{-4.2pt}\text{if $d_1 \not=0$} \vspace{5pt}\\
	\dfrac{1}{d_2}
	\left(
	\dfrac{\sigma}{\Gamma} - \dfrac{L_{\op}(1-\rho)}{1+\rho} 
	\right) & \hspace{-4.2pt}\text{if $d_1 =0$}.
\end{cases}
\]
Fix  a constant $\lambda$ as in Lemma~\ref{lem:Qlog_prop}, and
let $\widetilde \tau_{\min} \in \mathbb{R}_{>0}$ be
the solution of the equation
$
\lambda \left(
1+ \nu(t) 
\right) = 1,
$
that is, $\widetilde \tau_{\min}$ is defined as in \eqref{eq:tilde_tau_def}.

The following theorem shows that  
the norm $\|x(t)\|_{\cl}$ of the state trajectory starting in $\bcl(R / L_{\cl})$ 
decreases monotonically and 
exponentially without Zeno behavior.
\begin{theorem}
	\label{thm:log_case}
	Suppose that Assumptions~\ref{assump:closed}--\ref{assump:control_part} and 
	\ref{assump:initial_cond} hold. 
	If the quantization density $\rho \in (0,1)$
	and the threshold parameter $\sigma\in (0,\sigma_0]$ 
	satisfy
	\begin{equation}
		\label{eq:thres_cond}
		\frac{\Gamma L_{\op}(1-\rho)}{1+\rho} 
		< \sigma < \frac{2c\rho}{\alpha L_{\cl}(1+\rho)}
		\eqqcolon \sigma_1,
	\end{equation}
	then 
	the closed-loop system \eqref{eq:closed_loop} with STM~\eqref{eq:STM} has
	the following properties for every initial state 
	$x_0 \in \bcl( R / L_{\cl})$
	and upper bound $\tau_{\max} \in \mathbb{R}_{>0}$ of the
	inter-sampling times:
	\begin{enumerate}
		\renewcommand{\labelenumi}{(\roman{enumi})}
		\item
		The inter-sampling times satisfy
		\begin{equation}
			\label{eq:min_inter_event_time}
			t_{k+1} - t_k \geq \min\{ \tau_{\max},\,
			\tau_{\min} ,\, \widetilde \tau_{\min} 
			\} >0
		\end{equation}
		for all $k \in \mathbb{Z}_{\geq 0}$.
		\item
		There exists a unique solution $x$ of the closed-loop system 
		\eqref{eq:closed_loop} with STM~\eqref{eq:STM} on $\mathbb{R}_{\geq 0}$.
		\item 
		The solution $x$ satisfies 
		\[
		\|x(t)\|_{\cl} \leq 
		e^{-\gamma t} \|x_0\|_{\cl} 
		\]
		for all $t \in \mathbb{R}_{\geq 0}$, where
		\begin{equation}
			\label{eq:gamma_def}
			\gamma \coloneqq \frac{-1}{\tau_{\max}} \log\left( e^{-c\tau_{\max} }
			\left(
			1 - \frac{\sigma }{\sigma_1 }  
			\right) + 
			\frac{\sigma }{\sigma_1 }    \right)>0.
		\end{equation}
	\end{enumerate}
	\mbox{} \vspace{-5pt}
\end{theorem}

\begin{remark}[Parameter dependency]
	As the quantization density $\rho$
	increases, the range of the threshold parameter 
	$\sigma$ given by
	\eqref{eq:thres_cond} expands. 
	A large $\sigma $ leads to a decrease
	in sampling frequency. On the other hand, 
	as $\sigma$ becomes smaller, 
	the upper bound $\gamma$ of the decay rate
	given by \eqref{eq:gamma_def} increases to 
	the contraction rate $c$ in Assumption~\ref{assump:closed}, 
	and faster convergence
	can be expected.
	Note also that the upper bound $\tau_{\max} \in \mathbb{R}_{>0}$ of  the
	inter-sampling times affects only $\gamma$.
	\hspace*{\fill} $\triangle$
\end{remark}

\subsubsection{Proof of main result}
The proof of Theorem~\ref{thm:log_case} is based on three lemmas.
We begin 
with a  technical lemma, which gives  an upper 
bound of the decay rate in statement (iii) of 
Theorem \ref{thm:log_case}. This result has been used without detailed derivation in the proof of \cite[Theorem~4.1]{Wakaiki2018_EVC}.
For the sake of completeness, we provide all the details in Appendix.
\begin{lemma}
	\label{lem:decay_bound}
	Let $\varepsilon \in (0,1)$ and $c,\tau_{\max} \in \mathbb{R}_{>0}$.
	Define the functions $w$ and $W$ by
	\[
	w(t) \coloneqq e^{-ct} (1-\varepsilon ) + \varepsilon  \quad  \text{and} \quad 
	W(t) \coloneqq - \frac{\log w(t) }{t}
	\]
	for $t \in \mathbb{R}_{>0}$. 
	Set
	$\gamma \coloneqq W(\tau_{\max})$. Then
	$\gamma > 0$ and
	\begin{align}
		\label{eq:f_tmax_bound}
		w(t) \leq e^{-\gamma t} \quad \text{for all $t \in [0,\tau_{\max}]$.}
	\end{align}
	\mbox{} \vspace{-5pt}
\end{lemma}

Second, we focus on 
the first sampling time $t_1$ and 
show that statement (i) of Theorem~\ref{thm:log_case} is true in the case $k=0$. 
To this end, we use the condition \eqref{eq:rho_sigma_ineq}, i.e., 
the first inequality in \eqref{eq:thres_cond}.
\begin{lemma}
	\label{lem:mim_inter_event}
	Suppose that Assumptions~\ref{assump:closed}, \ref{assump:open}, and 
	\ref{assump:initial_cond} hold.
	Let 
	$q_0 \coloneqq Q_{\log}(x_0)$ for some
	$x_0  \in \bcl(R / L_{\cl})$.
	Assume that  the quantization density $\rho \in (0,1)$ and
	the threshold parameter $\sigma \in (0,\sigma_0]$
	satisfy the condition
	\eqref{eq:rho_sigma_ineq}.
	Then 
	the sampling time $t_1$ defined by the STM~\eqref{eq:STM} satisfies
	\begin{equation}
		\label{eq:t1_ineq}
		t_1 \geq \min\{ \tau_{\max},\,
		\tau_{\min},\, \widetilde \tau_{\min}
		\} > 0.
	\end{equation}
\end{lemma}
\begin{proof}
	Choose 
	a constant $\lambda$ as in Lemma~\ref{lem:Qlog_prop}.
	Then this lemma  shows that $
	q_0 \in  \bcl(\lambda R)
	$
	holds
	under 
	Assumption~\ref{assump:initial_cond}.
	By Lemma~\ref{lem:open_loop_bound},
	the ODE \eqref{eq:x_q_ODE} with $q=q_0$
	has the unique solution $x_{q_0}$ on $[ 0, \widetilde \tau_{\min} ]$, and
	$x_{q_0}(\tau) \in \bop(R_2)$ for all $
	\tau \in [ 0, \widetilde \tau_{\min} ]$.
	Hence, $\widetilde \tau_0$ given by \eqref{eq:log_SMT3} satisfies
	$\widetilde \tau_0 \geq \min\{ \tau_{\max}, \,\widetilde \tau_{\min}\}  > 0$.

	Next, we obtain a lower bound of $\tau_0$, where $\tau_0$ is 
	defined  by \eqref{eq:log_SMT2}.
	For all $\tau \in [0,  \widetilde \tau_0)$,
	Lemma~\ref{lem:is_diff} yields
	\[
	\|x_{q_0}(\tau) - q_0 \|_{\op}  \leq \nu(\tau) \|q_0\|_{\op}
	\]
	and hence $\psi_{\log}$ defined by \eqref{eq:psi_log} satisfies
	\begin{align*}
		\psi_{\log}(\tau,q_0) 
		\leq
		\left(\frac{L_{\op}(1-\rho)}{1+\rho} e^{d_1 \tau}+ 
		\nu(\tau)\right)   \|q_0\|_{\op}.
	\end{align*}
	Since 
	$
	\|q_0\|_{\op} \leq \Gamma \|q_0\|_{\cl},
	$
	the following implication holds  for all $\tau \in [0,  \widetilde \tau_0)$:
	\begin{align*}
		\Gamma
		&\left(\frac{L_{\op}(1-\rho)}{1+\rho} e^{d_1 \tau}+ 
		\nu(\tau)\right)   \|q_0\|_{\cl} \leq \sigma \|q_0\|_{\cl} \\
		&\qquad
		\Rightarrow
		\quad 
		\psi_{\log}(\tau, q_0) \leq \sigma \|q_0\|_{\cl}.
	\end{align*}
	Therefore, 
	$\tau_0 \geq \min\{ \widetilde \tau_0 ,\, \tau_{\min} \}$. 
	In addition,
	the inequality \eqref{eq:rho_sigma_ineq} yields
	$\tau_{\min} > 0$. 
	Thus,
	we obtain the desired conclusion \eqref{eq:t1_ineq}.
\end{proof}

Third, we show that statements (ii) and (iii) 
of Theorem~\ref{thm:log_case} are
true on the first sampling interval $[0,t_1]$.
The stability property follows from the condition $\sigma < \sigma_1$, i.e., the second inequality in \eqref{eq:thres_cond}. 
\begin{lemma}
	\label{lem:first_conv}
	Suppose that Assumptions~\ref{assump:closed}--\ref{assump:control_part}
	and \ref{assump:initial_cond} hold. 
	Assume that the quantization density $\rho \in (0,1)$ and
	the threshold parameter $\sigma\in (0,\sigma_0]$ 
	satisfy the condition
	\eqref{eq:thres_cond}.
	Then 
	there exists a unique 
	solution $x$ of the following ODE  on $[0,t_1]$:
	\begin{equation}
		\label{eq:log_ODE}
		\dot x(t) = f\big(x(t),g(q_0)\big),\quad x(0) = x_0 \in   \bcl( R / L_{\cl}),
	\end{equation}
	where $q_0 \coloneqq Q_{\log}(x_0)$ and $t_1 \in \mathbb{R}_{>0}$ is defined by
	the STM~\eqref{eq:STM}.
	Furthermore,
	the solution $x$ satisfies
	\begin{equation}
		\label{eq:x_conv_log}
		\|x(t)\|_{\cl} \leq 
		e^{-\gamma t} \|x_0\|_{\cl} 
	\end{equation}
	for all $t \in [0 , t_1]$.	
\end{lemma}
\begin{proof}
	The first sampling time $t_1$
	satisfies $t_1 >0$ by Lemma~\ref{lem:mim_inter_event} and
	$t_1 \leq \tau_{\max}$ by definition.
	We start by showing that 
	the ODE \eqref{eq:log_ODE} has a unique solution $x$ on $[0,t_1]$ and that $x$ satisfies
	\begin{equation}
		\label{eq:xt_x0_bounded}
		\|x(t)\|_{\cl} \leq \|x_0\|_{\cl} \quad \text{for all
			$t \in
			[0, t_1 )$.}
	\end{equation}
	Assume, to get a contradiction, that 
	\begin{itemize}
		\item
		the solution $x$ of 
		the ODE \eqref{eq:log_ODE}
		either does not exist or is not unique 
		on
		$
		[0, t_1]
		$; or that
		\item 
		there exists a unique solution $x$
		of the ODE \eqref{eq:log_ODE}
		on
		$
		[0, t_1]
		$, but
		the solution $x$ does not satisfy the inequality \eqref{eq:xt_x0_bounded}.
	\end{itemize}
	In both scenarios, there exists $s_0 \in (0,t_1)$ such that 
	the ODE \eqref{eq:log_ODE} has 
	a unique solution $x$
	on
	$
	[0, s_0]
	$ satisfying
	$ \|x(s_0)\|_{\cl}  > \|x_0\|_{\cl} $.
	Define
	\[
	s_1 \coloneqq 
	\inf\{t \in \mathbb{R}_{> 0}:
	\|x(t)\|_{\cl} > \|x_0\|_{\cl} 
	\} \in [0,s_0].
	\]
	From the continuity of $x$, we obtain $s_1 < s_0$ and
	\[
	\|x(s_1)\|_{\cl} = 
	\|x_0\|_{\cl} < R.
	\] 
	Hence,  there exists $\delta \in (0, s_0 - s_1)$ such that 
	for all $t \in [0, s_1+\delta)$,
	\begin{equation}
		\label{eq:x_bopR2}
		x(t) \in \bcl(R) \subseteq \bop(R_2).
	\end{equation}
	The unique solution $x_{q_0}(t)$
	of the ODE~\eqref{eq:x_q_ODE}
	with $q = q_0$ exists and also satisfies
	$x_{q_0}(t) \in \bop(R_2)$ 
	for all $t \in [0, t_1)$ by the definition \eqref{eq:log_SMT3} of $\widetilde \tau_1$
	and the inequality $\widetilde \tau_1 \geq \tau_1 = t_1$.
	
	Define $e(t) \coloneqq q_0- x(t)$ for $t \in [0, s_1+\delta)$. 
	Since 
	\[
	x(t), x_{q_0}(t)  \in \bop(R_2)
	\] 
	for all $t \in [0, s_1+\delta)$, 
	Lemma~\ref{lem:x_q_diff} and the STM~\eqref{eq:STM} yield
	\begin{equation}
		\label{eq:error_log}
		\| e(t) \|_{\op} \leq \psi_{\log}(t,q_0) \leq \sigma\|q_0\|_{\cl}
	\end{equation}
	for all $t \in [0, s_1+\delta)$.
	Combining this with 
	\[
	\|q_0\|_{\cl} < R \quad  \text{and}
	\quad 
	\sigma \leq \sigma_0,
	\]
	we have 
	$e(t) \in \bop(\sigma_0R)$ for all $t \in [0, s_1+\delta)$.

	Recall that $x$ is the solution of the ODE
	\[
	\dot x(t) = f\big(x(t),g(q_0)\big) = F\big(x(t),e(t)\big),\quad 
	x(0) = x_0.
	\]
	on $[0,  s_0]$.
	Since
	$x(t) \in \bcl(R)$ for all $t \in [0, s_1+\delta)$, we see from
	Theorem~\ref{thm:basic_prop_of_contraction3}
	and the inequality \eqref{eq:error_log} that 
	under  Assumptions~\ref{assump:closed} and \ref{assump:control_part},
	\begin{align*}
		\|
		x(t)
		\|_{\cl} 
		&\leq e^{-ct} \|x_0\|_{\cl}  + \frac{\sigma  \alpha(1-e^{-ct})}{c} \|q_0\|_{\cl}
	\end{align*}
	holds
	for all $t \in [0, s_1+\delta]$.
	Applying the property \eqref{eq:log_prop2} of the logarithmic quantizer,
	we obtain
	\begin{align}
		\label{eq:xt_bounded_initial}
		\|
		x(t)
		\|_{\cl} 
		&\leq  	\left(
		e^{-ct}
		\left(
		1 - \frac{\sigma}{\sigma_1}
		\right) + 
		\frac{\sigma}{\sigma_1}
		\right) \|x_0\|_{\cl} 
	\end{align}
	for all $t \in [0, s_1+\delta]$, where $\sigma_1$ is as in
	\eqref{eq:thres_cond}.
	From the condition \eqref{eq:thres_cond}, we have $\sigma / \sigma_1 < 1$. Hence 
	$
	\|
	x(t)
	\|_{\cl} \leq \|x_0\|_{\cl}
	$ 
	for all $t \in [0, s_1+\delta]$.
	This contradicts the definition of $s_1$.

	We have shown  that 
	a unique solution $x$
	of the ODE \eqref{eq:log_ODE}
	exists on
	$
	[0, t_1]
	$ and satisfies the inequality \eqref{eq:xt_x0_bounded}.
	Then \eqref{eq:x_bopR2} also holds for all $t \in [0, t_1)$.
	Therefore, we can
	replace $s_1 + \delta $ by $t_1$ in the above argument, and
	the inequality
	\eqref{eq:xt_bounded_initial} is satisfied for all $t \in [0, t_1]$.
	By Lemma~\ref{lem:decay_bound}, the constant $\gamma $ defined 
	by \eqref{eq:gamma_def} satisfies
	\[
	e^{-ct}
	\left(
	1 - \frac{\sigma}{\sigma_1}
	\right) + 
	\frac{\sigma}{\sigma_1}
	\leq e^{-\gamma t}
	\]
	for all $t \in [0, \tau_{\max}]$.
	Thus,
	the desired inequality \eqref{eq:x_conv_log} holds for all $t \in [0, t_1]$.
\end{proof}

After these preparations,
we are now ready to prove Theorem~\ref{thm:log_case}.

\noindent\hspace{1em}{\textit{Proof of Theorem~\ref{thm:log_case}.} }
Let $x_0 \in \bcl(R / L_{\cl})$. 
By Lemma~\ref{lem:first_conv}, 
\[
\|x(t_1)\|_{\cl} \leq e^{-\gamma t_1} \|x_0\|_{\cl},
\]
and hence $x(t_1) \in \bcl(R / L_{\cl})$. 
Lemma~\ref{lem:mim_inter_event} shows that 
\[
t_2 - t_1 \geq \min\{ \tau_{\max},\,
\tau_{\min},\, \widetilde \tau_{\min}
\} > 0.
\] 
Using Lemma~\ref{lem:first_conv} again, we have that
there exists a unique solution $x$ of the closed-loop system \eqref{eq:closed_loop}
on $[0,t_2]$ and that $x$ satisfies
\[
\|x(t)\|_{\cl} \leq e^{-\gamma (t-t_1)} \|x(t_1)\|_{\cl}
\]
for all $t \in [t_1, t_2]$.
Repeating this argument yields the desired conclusions.
\hspace*{\fill} $\blacksquare$

We conclude this section by making a  remark on discretization of inter-sampling
times, which is applied when the STM sends inter-sampling times through
communication channels.
\begin{remark}[Discretization of inter-sampling times]
	The inter-sampling times generated by the STM \eqref{eq:STM} can take values
	on the interval $[\min \{ \tau_{\min},\,\widetilde \tau_{\min} \}, \tau_{\max}]$.
	The proposed STM can be easily modified so that 
	inter-sampling times belong to a finite or countable set 
	$S \subset \mathbb{R}_{>0}$.
	Assume that 
	\[
	0< \inf S \leq 
	\min \{ \tau_{\min},\,\widetilde \tau_{\min} \} \leq \sup S \leq \tau_{\max}.
	\]
	We define the new $k$-th inter-sampling time $\tau_{\text{dis},k}$ by
	\[
	\tau_{\text{dis},k} \coloneqq \sup \{ \tau \in S : \tau \leq \tau_k\},
	\]
	where $\tau_k$ is determined by the STM \eqref{eq:STM}.
	Then 
	the argument in this section shows that the same result as in
	Theorem~\ref{thm:log_case} holds for
	the discretized inter-sampling times
	$\{ \tau_{\text{dis},k} \}_{k \in \mathbb{Z}_{\geq 0}}$, although the sampling frequency increases.
	\hspace*{\fill} $\triangle$ 
\end{remark}

\section{Self-triggered stabilization under zooming quantization}
\label{sec:zoom_case}
In this section, we study the co-design of 
zooming quantization and self-triggered sampling 
for contracting systems.
First, we briefly explain the zooming quantizer introduced in \cite{Brockett2000,Liberzon2003Automatica}
and make an assumption on the initial zoom parameter.
Second, we propose an STM that uses the zoom parameter for 
a threshold of measurement errors.
Finally, we present an update rule for the zoom parameter and
give a sufficient condition for stabilization.

\subsection{System model}
\subsubsection{Zooming quantization}
Let $M,\Delta \in \mathbb{R}_{>0}$. 
Using the norms in Assumptions~\ref{assump:closed} and \ref{assump:open},
we assume that 
the quantization function $Q \colon \mathbb{R}^n \to \mathbb{R}^n$
satisfies $\|Q(x) - x\|_{\op} \leq \Delta$ if $\|x\|_{\cl} < M$.
Notice that we use the norm $\|\cdot\|_{\cl}$ for quantization ranges but
the norm $\|\cdot\|_{\op}$ for quantization errors.

For a fixed $\mu \in \mathbb{R}_{>0}$,
define the function $Q_{\mu}$ by
\[
Q_\mu(x) \coloneqq \mu Q\left(\frac{x}{\mu} \right),\quad 
x \in \mathbb{R}^n.
\]
Then $Q_{\mu}$ satisfies
\begin{equation}
	\label{eq:Qmu_cond}
	\|Q_{\mu}(x) - x \|_{\op} \leq \Delta \mu\quad \text{for all $x \in \bcl(M\mu)$}.
\end{equation}
We call $\mu$ {\em the zoom parameter}. We make an assumption on an initial zoom parameter.
\begin{assumption}
	\label{assump:initial}
	The initial zoom parameter $\mu_0 \in \mathbb{R}_{>0}$ satisfies
	\begin{equation}
		\label{eq:mu0_def}
		\mu_0 < \frac{R}{M+\Gamma \Delta},
	\end{equation}
	where the constants $\Gamma, R \in \mathbb{R}_{>0}$ are  as in
	Section~\ref{sec:nonlinear_assump}.
\end{assumption}

The following lemma shows that 
the quantized value of $x \in \bcl(M\mu_0)$ belongs to a ball
strictly smaller 
than $\bcl(R)$ under Assumption~\ref{assump:initial}.
\begin{lemma}
	\label{lem:lambda_choice}
	Suppose that Assumptions~\ref{assump:closed}, \ref{assump:open}, and \ref{assump:initial} hold, and let
	$\mu \in (0, \mu_0]$.
	Then the constant $\lambda$ defined by
	\begin{equation}
		\label{eq:lambda_def_TV}
		\lambda \coloneqq \frac{\mu_0}{R} (M+\Gamma \Delta) 
	\end{equation}
	satisfies $\lambda < 1$. Moreover,
	$Q_{\mu}(x) \in \bcl(\lambda R)$
	for all $x\in \bcl(M\mu)$.
\end{lemma}
\begin{proof}
	The inequality $\lambda < 1$ immediately follows from
	\eqref{eq:mu0_def}.
	Let $x\in \bcl(M\mu)$ for some $\mu\in\mathbb{R}_{>0}$. By
	\eqref{eq:Qmu_cond},
	\begin{equation*}
		\|Q_{\mu}(x) \|_{\cl} \leq 
		\|Q_{\mu}(x) - x\|_{\cl} + \|x\|_{\cl} <
		\Gamma \Delta \mu+ M\mu.
	\end{equation*}
	If $\mu \leq \mu_0$, then
	this and \eqref{eq:lambda_def_TV} 
	yield $\|Q_{\mu}(x) \|_{\cl} <  \lambda R$.
\end{proof}

\subsubsection{Self-triggered control system}
Let 
$\{t_k\}_{k \in \mathbb{Z}_{\geq 0} } 
$  be a strictly increasing sequence with $t_0 \coloneqq 0$.
We consider the following closed-loop system:
\begin{equation}
	\label{eq:closed_TV}
	\left\{
	\begin{aligned}
		\dot x(t) &= f\big(x(t),u(t)\big), & & t \in \mathbb{R}_{\geq 0};\qquad x(0) = x_0\\
		u(t) &= g(q_k),& &
		t \in [t_k, t_{k+1}),~k \in \mathbb{Z}_{\geq 0} \\
		q_k &= Q_{\mu_k}\big(x(t_k)\big),& & k \in \mathbb{Z}_{\geq 0},
	\end{aligned}
	\right.
\end{equation}
where $\mu_k \in \mathbb{R}_{>0}$ is the $k$-th zoom parameter for $k \in \mathbb{Z}_{\geq 0}$.

In this section, 
the sampling times $\{t_k\}_{k \in \mathbb{Z}_{\geq 0} }$  are
chosen from the discrete 
set $\{p h: p \in \mathbb{Z}_{\geq 0} \}$ for some
$h \in \mathbb{R}_{> 0}$. We refer to $h$ as the period for the STM.
Let $\ell_{\max} \in \mathbb{N}$ and
\[
\floor_h(\tau) \coloneqq \max\{\ell \in \mathbb{Z}: \ell h \leq \tau \},
\] 
for $\tau \in \mathbb{R}$,
that is,
$\floor_h(\tau) $ is the greatest integer less than or equal to $\tau / h $.
The inequalities \eqref{eq:x_q_diff}  and 
\eqref{eq:Qmu_cond} motivates us 
to define the function $\psi_{\text{zo}}$ by
\begin{equation}
	\label{eq:psi_dyn}
	\psi_{\text{zo}}(\tau, q, \mu) \coloneqq  
	\Delta \mu  e^{d_1 \tau}+  
	\|x_q(\tau) - q \|_{\op}
\end{equation}
for $\tau\in \mathbb{R}_{\geq0}$ and $q \in \bcl(R)$
such that there exists a unique solution $x_q$ of the ODE \eqref{eq:x_q_ODE} on $[0,\tau]$.
We employ the following STM:
\begin{subequations}
	\label{eq:STM_tm}
	\begin{empheq}[left = {\empheqlbrace \,}, right = {}]{align}
		t_{k+1} &\coloneqq t_k + \ell_k h \label{eq:tv_SMT1}\\
		\ell_k &\coloneqq 
		\begin{cases}
			\widetilde \ell_k&
			\text{if $\psi_{\text{zo}}(\tau, q_k, \mu_k) \leq \sigma M\mu_k$
				for all $\tau \in (h, \widetilde \ell_k h)$} \\
			\floor_h(\inf \{
			\tau >h:
			\psi_{\text{zo}}( \tau, q_k, \mu_k) > \sigma M \mu_k 
			\})& \text{otherwise}  
			\label{eq:tv_SMT2}
		\end{cases} \\
		\widetilde \ell_k &\coloneqq 
		\begin{cases}
			\ell_{\max} &
			\text{if $\|x_{q_k}(\tau) \|_{\op} < R_2$
				for all $\tau \in (h, \ell_{\max} h)$} \\
			\floor_h (\inf \{
			\tau >h :
			\|x_{q_k}(\tau)  \|_{\op} \geq R_2
			\}) \hspace{26pt} & \text{otherwise}, \label{eq:tv_SMT3}
		\end{cases}
	\end{empheq}
\end{subequations}
where  $x_{q_k}$ is the solution of the ODE \eqref{eq:x_q_ODE}  with $q=q_k$.
Since $1 \leq \ell_k \leq \widetilde \ell_k \leq \ell_{\max}$
by construction, we have 
\[
h \leq t_{k+1} - t_k \leq \ell_{\max}h
\]
for all $k \in \mathbb{Z}_{\geq 0}$.
Hence Zeno behavior does not occur.
Fig.~\ref{fig:closed_loop2} illustrates the closed-loop system
and the data flow in the networked control framework.
Note that 
$\ell_k$ is transmitted not 
only to the sensor but also the encoder and the decoders, because
$\ell_k$ 
is used to update
the zoom parameter $\mu_k$.

\begin{figure}[bt]
	\centering
	\includegraphics[width = 7.5cm]{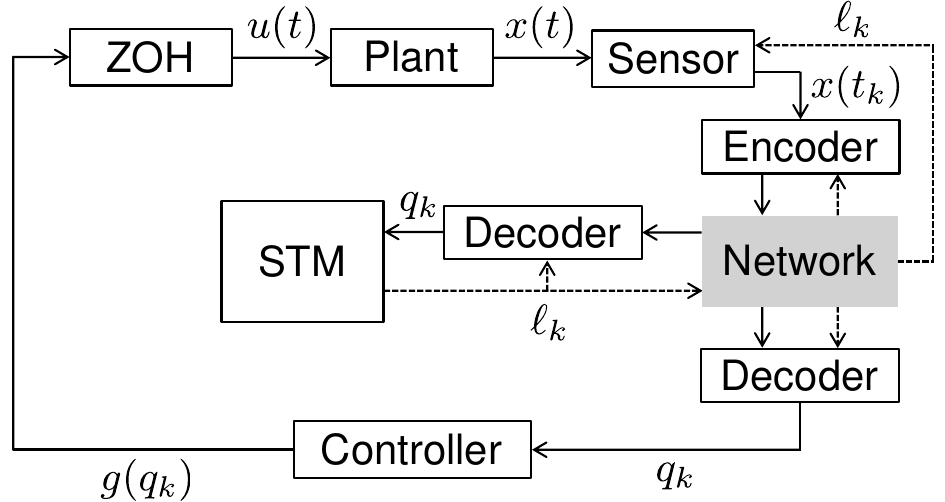}
	\caption{Closed-loop system in networked control framework. Encoding and decoding of $\ell_k$ are omitted for brevity.}
	\label{fig:closed_loop2}
\end{figure}

\stepcounter{equation}
The STM~\eqref{eq:STM_tm} aims at guaranteeing that the measurement error
$\| x(t_k+\tau) -q_k\|_{\op}$ is upper-bounded by
$\sigma M \mu_k$ for all $\tau \in [0, t_{k+1} - t_k)$ and 
$k \in \mathbb{Z}_{\geq 0}$. 
The roles of the triggering conditions in \eqref{eq:tv_SMT2}
and \eqref{eq:tv_SMT3} are the same as those in the case of 
logarithmic quantization.
In fact,
\eqref{eq:tv_SMT2}
yields
\begin{equation}
	\label{eq:psi_zo_bound_explanation}
	\psi_{\text{zo}}(\tau, q_k, \mu_k) \leq \sigma  M \mu_k
\end{equation}
for all $\tau \in (h, \ell_kh)$.
If 
$h$ is sufficiently small, then
the ODE \eqref{eq:x_q_ODE} with $q=q_k$
has a unique solution $x_{q_k}$ on $[0,h]$ satisfying
$x_{q_k}(\tau) \in \bop(R_2)$ for all $\tau \in [0,h]$, and
the inequality \eqref{eq:psi_zo_bound_explanation} 
holds on $[0,h]$,
as will be discussed in Lemma~\ref{lem:0h_bound} below.
By \eqref{eq:tv_SMT3},
the solution $x_{q_k}$ can be continued to 
$[0,\widetilde \ell_k h]$ uniquely, and 
furthermore,
$x_{q_k} (\tau) \in \bop(R_2)$ 
for all $\tau \in [0, \widetilde \ell_k h)$.
From this fact and Lemma~\ref{lem:x_q_diff}, we will see that
\[
\|x(t_k+\tau) - q_k\|_{\op} \leq \psi_{\text{zo}}(\tau, q_k, \mu_k) 
\]
for all $\tau \in [0, \ell_kh)$.

\subsection{Stability analysis}
\subsubsection{Main result in zooming quantization case}
We present the main result of this section.
The following theorem 
gives an update rule for the zoom parameter $\mu_k$,
which guarantees that 
state trajectories starting in $\bcl(M\mu_0)$ exponentially converge
to the origin under the STM~\eqref{eq:STM_tm}.
\begin{theorem}
	\label{thm:varying_case}
	Suppose that Assumptions~\ref{assump:closed}--\ref{assump:control_part} and \ref{assump:initial}  hold, and define $\lambda \in (0,1)$ and $\widetilde \tau_{\min} \in \mathbb{R}_{>0}$
	as in \eqref{eq:lambda_def_TV} and \eqref{eq:tilde_tau_def}, respectively. Assume that
	the parameters $(M,\Delta) \in \mathbb{R}_{>0}^2$ of quantization and 
	the parameters $(h,\sigma) \in (0,\widetilde \tau_{\min}]
	\times (0,\sigma_0]$ of self-triggered sampling satisfy
	\begin{align}
		\label{eq:thres_cond_varying}
		\frac{\Delta}{M} \big(e^{d_1 h}+\nu(h) \big) + \Gamma \nu(h) 
		\leq 
		\sigma < \frac{c}{\alpha}.
	\end{align}
	If the zoom parameter $\mu_k$ is defined by
	\begin{equation}
		\label{eq:zoom_parameter_update}
		\mu_{k+1} \coloneqq \left(
		e^{-c(t_{k+1}-t_k)}
		\left(
		1 - \frac{\alpha  \sigma}{c}
		\right) + 
		\frac{\alpha  \sigma}{c}
		\right) \mu_k
	\end{equation}
	for $k \in \mathbb{Z}_{\geq 0}$,
	then the closed-loop system \eqref{eq:closed_TV} with STM~\eqref{eq:STM_tm} has
	the following properties
	for every initial state $x_0 \in \bcl (M \mu_0)$
	and upper bound $\ell_{\max} \in \mathbb{N}$ of $\{\ell_k \}_{k \in \mathbb{Z}_{\geq 0}}$:
	\begin{enumerate}
		\renewcommand{\labelenumi}{(\roman{enumi})}
		\item There exists a unique solution $x$ of the closed-loop system \eqref{eq:closed_TV} with STM~\eqref{eq:STM_tm} on $\mathbb{R}_{\geq 0}$.
		\item
		The solution $x$ satisfies
		\[
		\|x(t)\|_{\cl} < 
		M\mu_0  e^{-\gamma t} 
		\]
		for all $t \in \mathbb{R}_{\geq 0}$, where
		\begin{equation}
			\label{eq:bound_decay_rate_zo}
			\gamma \coloneqq \frac{-1}{\ell_{\max} h }
			\log
			\left( e^{-c\ell_{\max} h}
			\left(
			1 - \frac{\alpha\sigma }{c }  
			\right) + 
			\frac{\alpha\sigma }{c }   \right)>0.
		\end{equation}
	\end{enumerate}
\end{theorem}

In Theorem~\ref{thm:varying_case},
the state norm $\|x(t)\|_{\cl}$ is upper-bounded
by using the initial state bound $M\mu_0$, instead of
the initial state norm $\|x_0\|_{\cl}$.
Consequently,
the decrease of $\|x(t)\|_{\cl}$ might not be monotonic 
in the zooming quantization case, unlike the logarithmic quantization case.

The zooming quantizer and 
the STM \eqref{eq:STM_tm} interact with each other.
In fact,
the zoom parameter of the quantizer decreases
depending on inter-sampling times; see \eqref{eq:zoom_parameter_update}.
On the other hand, the STM \eqref{eq:STM_tm} uses the zoom parameter
for a threshold.

\begin{remark}[Parameter dependency]
	As the  number 
	of quantization levels increases, 
	the fraction
	$\Delta / M$ in 
	\eqref{eq:thres_cond_varying} becomes smaller.
	Therefore, 
	a smaller threshold parameter $\sigma$
	can be chosen when finer quantization is applied.
	Moreover, as $\sigma$ becomes smaller, 
	the zoom parameter $\mu_k$ and therefore quantization errors
	decrease faster. 
	The period $h$ for the STM  has to satisfy $h \leq \widetilde \tau_{\min}$.
	This upper bound $\widetilde \tau_{\min}$ becomes larger
	as 
	the quantization parameters $M$, $\Delta$, and $\mu_0$  decrease, 
	where
	we used the relation between $\lambda$ and the quantization parameters; see
	\eqref{eq:lambda_def_TV}.
	Note that, as
	$M$ and $\mu_0$ decrease,
	the region $\bcl (M \mu_0)$ for initial states shrinks.
	As in the logarithmic quantization case,
	the upper bound $\ell_{\max}$ of $\{\ell_k \}_{k \in \mathbb{Z}_{\geq 0}}$
	appears only in the definition \eqref{eq:bound_decay_rate_zo}
	of the upper bound $\gamma$ on the 
	decay rate. \hspace*{\fill} $\triangle$
\end{remark}

\subsubsection{Proof of main result}
To prove Theorem~\ref{thm:varying_case},
we need two lemmas. We first
note that the  STM~\eqref{eq:STM_tm} does not check the conditions
for $\tau \in [0, h]$.
This is because the triggering conditions are not satisfied  on the interval $[0,h]$ if 
the first inequality in \eqref{eq:thres_cond_varying} holds, as shown in 
the following lemma.
\begin{lemma}
	\label{lem:0h_bound}
	Suppose that Assumptions~\ref{assump:closed} and \ref{assump:open}  hold.
	Let $\lambda \in (0,1)$, and define 
	$\widetilde \tau_{\min} \in \mathbb{R}_{>0}$
	by \eqref{eq:tilde_tau_def}. 
	Assume that 
	the parameters $(M,\Delta) \in \mathbb{R}_{>0}^2$ of quantization and 
	the parameters $(h,\sigma) \in  (0,\widetilde \tau_{\min}]
	\times (0,\sigma_0]$ of self-triggered sampling satisfy 
	\begin{equation}
		\label{eq:sigma_lower_bound}
		\frac{\Delta}{M} \big(e^{d_1 h}+\nu(h)\big) + \Gamma \nu(h) 
		\leq 
		\sigma.
	\end{equation}
	Then the following statements hold for all $q \in \bcl(\lambda R)$:
	\begin{enumerate} 
		\renewcommand{\labelenumi}{(\roman{enumi})}
		\item 
		There exists a unique solution $x_q$ of the ODE \eqref{eq:x_q_ODE} on $[0,h]$,
		and 
		$x_{q}(\tau)  \in \bop(R_2)$ for all 
		$\tau \in [0, h]$.
		\item 
		Let $\mu \in \mathbb{R}_{>0}$. 
		If $q = Q_{\mu}(x)$ for some $x \in \bcl(M\mu)$, then
		\[
		\psi_{\text{zo}}(\tau,q,\mu) < \sigma M \mu
		\]
		for all $\tau \in [0, h]$.
	\end{enumerate}
\end{lemma}
\begin{proof}
	Statement (i) directly follows from 
	Lemma~\ref{lem:open_loop_bound} and the inequality $h \leq \widetilde \tau_{\min}$. 
	To show statement (ii), we apply Lemma~\ref{lem:is_diff}
	and then obtain
	\[
	\psi_{\text{zo}}(\tau, q,\mu) \leq 
	\Delta \mu e^{d_1 h} + \nu(h) \|q\|_{\op}
	\]
	for all $\tau \in [0, h]$.
	Moreover, if $q = Q_{\mu}(x)$ for some $x \in \bcl(M\mu)$, then
	the inequality \eqref{eq:Qmu_cond} yields
	\begin{align*}
		\|q\|_{\op} \leq \|q - x\|_{\op} + \|x\|_{\op}
		< \Delta \mu + \Gamma M \mu. 
	\end{align*}
	Therefore, 
	\begin{equation*}
		\psi_{\text{zo}}(\tau, q,\mu) <  \Big(\Delta \big(e^{d_1 h}+\nu(h)\big) 
		+ \Gamma M \nu(h)\Big) \mu
	\end{equation*}
	for all $\tau \in [0, h]$. Combining this with the inequality \eqref{eq:sigma_lower_bound},
	we see that statement (ii) holds.
\end{proof}

Next, we show
that the state $x$ 
has an upper bound that decreases on $[t_k,t_{k+1}]$ for  $k \in \mathbb{Z}_{\geq 0}$.
To this end, we use the condition $\sigma < c / \alpha$,
i.e., the second inequality in \eqref{eq:thres_cond_varying}.
\begin{lemma}
	\label{lem:conv_TV_kth}
	Suppose that Assumptions~\ref{assump:closed}--\ref{assump:control_part} and
	\ref{assump:initial} hold, and 
	define 
	$\widetilde \tau_{\min} \in \mathbb{R}_{>0}$
	by \eqref{eq:tilde_tau_def}, where $\lambda \in (0,1)$ is as in 
	\eqref{eq:lambda_def_TV}.
	Assume that the parameters $(M,\Delta) \in \mathbb{R}_{>0}^2$ of quantization and 
	the parameters $(h,\sigma) \in (0,\widetilde \tau_{\min}]
	\times (0,\sigma_0]$ of self-triggered sampling satisfy
	the condition
	\eqref{eq:thres_cond_varying}.
	Let $k \in \mathbb{Z}_{\geq 0}$ and $t_k \in \mathbb{R}_{\geq 0}$.
	If $x_k \in \bcl(M\mu_k)$ for some $\mu_k \in (0, \mu_0]$, then
	the ODE
	\begin{equation}
		\label{eq:closed_loop_kth}
		\left\{
		\begin{aligned}
			\dot x(t) &= f\big(x(t),g(q_k) \big),\quad x(t_k) = x_k\\
			q_k &= Q_{\mu_k}\big(x(t_k)\big)
		\end{aligned}
		\right.
	\end{equation}
	has a unique solution on $[t_k,t_{k+1}]$, 
	where $t_{k+1}$ is determined by the STM \eqref{eq:STM_tm}.
	Furthermore, the solution $x$ satisfies
	\begin{equation}
		\label{eq:state_decrease_TV_kth}
		\|x(t)\|_{\cl} <
		M\mu_k\left(
		e^{-c(t-t_k)}
		\left(
		1 - \frac{\alpha  \sigma}{c}
		\right) + 
		\frac{\alpha  \sigma}{c}
		\right) 
	\end{equation}
	for all $t \in [t_k,t_{k+1}]$. 
\end{lemma}
\begin{proof}
	Let $x_k \in \bcl(M\mu_k)$ for some $\mu_k\in  (0, \mu_0]$.
	First, we prove that the ODE \eqref{eq:closed_loop_kth}
	has
	a unique solution $x$ on $[t_k,t_{k+1}]$
	and that $x$ satisfies
	\begin{equation}
		\label{eq:x_Mmu_k}
		\|x(t)\|_{\cl} < M\mu_k \quad \text{for all $t \in [t_k, t_{k+1})$}.
	\end{equation}
	Assume, to get a contradiction, that 
	\begin{itemize}
		\item
		the solution $x$
		of the ODE \eqref{eq:closed_loop_kth} 
		either does not exist or is not unique 
		on $[t_k,t_{k+1}]$; or that
		\item 
		the ODE \eqref{eq:closed_loop_kth} 
		has a unique solution $x$ on $[t_k,t_{k+1}]$, but
		$x$ does not satisfy the inequality \eqref{eq:x_Mmu_k}.
	\end{itemize}
	In both cases, there exists $s_0 \in (t_k,t_{k+1})$ such that 
	the ODE \eqref{eq:closed_loop_kth} has a unique 
	solution $x$ on $[t_k,s_0]$ satisfying
	$\|x(s_0)\|_{\cl} \geq  M\mu_k $.
	Define
	\[
	s_1 \coloneqq \inf\{
	t >t_k :
	\|x(t)\|_{\cl} \geq M\mu_k
	\} \in (t_k, s_0].
	\]
	The continuity of $x$ yields
	\begin{equation}
		\label{eq:xs1}
		\|x(s_1)\|_{\cl} = M\mu_k.
	\end{equation}
	
	From $x_k \in \bcl(M\mu_k)$ and 
	Lemma~\ref{lem:lambda_choice}, we have that
	$q_k \in \bcl(\lambda R)$ under
	Assumption~\ref{assump:initial}.
	Therefore, the unique solution $x_{q_k}(\tau)$ of the ODE \eqref{eq:x_q_ODE} with $q = q_k$ exists and satisfies
	\[
	x_{q_k}(\tau)  \in \bop(R_2)
	\] 
	for all $\tau \in [0, h]$ by statement (i) of Lemma~\ref{lem:0h_bound} and for all
	$\tau \in (h, t_{k+1} - t_k)$ by
	the definition \eqref{eq:tv_SMT3} of $\widetilde \ell_k$.
	
	Define 
	$e(t) \coloneqq q_k- x(t)$ for $t \in [t_k,  s_1)$. 
	Since 
	\[
	x(t),x_{q_k}(t-t_k)  \in \bop(R_2)
	\]
	for all $t \in [t_k, s_1 )$, we see  from 
	Lemma~\ref{lem:x_q_diff} that 
	\[
	\| e(t)  \|_{\op} \leq \psi_{\text{zo}}(t-t_k,q_k,\mu_k)
	\]
	for all $t \in [t_k, s_1 )$.
	Statement (ii) of Lemma~\ref{lem:0h_bound} shows that
	\begin{equation}
		\label{eq:psi_bound_k}
		\psi_{\text{zo}}(t-t_k,q_k,\mu_k) \leq \sigma M \mu_k
	\end{equation}
	for all $t \in [t_k, t_k+h]$. If $t_k+h< s_1 $, then this 
	inequality~\eqref{eq:psi_bound_k}
	holds also for all $t \in (t_k+h, s_1)$ by
	the definition \eqref{eq:tv_SMT2} of $\ell_k$.
	Hence
	\begin{equation}
		\label{eq:error_bound_zoom}
		\|e(t) \|_{\op} \leq \sigma M \mu_k 
	\end{equation}
	for all $t \in [t_k, s_1)$.
	In particular, using
	\[
	M\mu_k \leq M\mu_0 < R \quad  \text{and}
	\quad 
	\sigma \leq \sigma_0,
	\]
	we obtain $e(t) \in \bop(\sigma_0R)$ for all $t \in [t_k, s_1)$.
	
	Since $x$ is the solution of the ODE
	\[
	\dot x(t) = f\big(x(t),g(q_k)\big) = F\big(x(t),e(t)\big),\quad 
	x(t_k) = x_k
	\]
	on $[t_k, s_0]$,
	Theorem~\ref{thm:basic_prop_of_contraction3} 
	and the inequality~\eqref{eq:error_bound_zoom}
	show that 
	\begin{align*}
		\|x(t )\|_{\cl} 
		&\leq e^{-c(t-t_k) } \|x_k\|_{\cl}  +\alpha  \sigma M
		\mu_k \frac{1-e^{-c(t-t_k)}}{c} 
	\end{align*}
	holds
	for all $t \in [t_k, s_1]$ 	under Assumptions~\ref{assump:closed} and \ref{assump:control_part}. 
	From $\|x_k\|_{\cl} < M\mu_k$, we have that
	\begin{align}
		\|x(t )\|_{\cl}  < M\mu_k&\left(
		e^{-c (t-t_k)}
		\left(
		1 - \frac{\alpha  \sigma}{c}
		\right) + 
		\frac{\alpha  \sigma}{c}
		\right) \quad \text{for all $t \in [t_k, s_1]$}.
		\label{eq:x_tk_tau}
	\end{align}
	The condition \eqref{eq:thres_cond_varying}
	gives $\alpha \sigma /c <1$, and hence
	\[
	\|x(s_1)\|_{\cl} < M\mu_k,
	\]
	which contradicts \eqref{eq:xs1}.
	
	We have shown that the ODE \eqref{eq:closed_loop_kth}
	has a unique solution $x$
	on $[t_k,t_{k+1}]$ and that $x$ satisfies
	the inequality \eqref{eq:x_Mmu_k}.
	Replacing $s_1$ by $t_{k+1}$ in the above argument, we conclude from \eqref{eq:x_tk_tau} that
	the desired inequality \eqref{eq:state_decrease_TV_kth}
	is satisfied 
	for all $t \in [t_k, t_{k+1}]$.
\end{proof}

Now we are in the position to prove Theorem~\ref{thm:varying_case}.

\noindent\hspace{1em}{ {\textit{Proof of Theorem~\ref{thm:varying_case}.}} }
Let
$x(0) = x_0 \in \bcl(M\mu_0)$. By Lemma~\ref{lem:conv_TV_kth},
there exists a unique solution $x$ of the closed-loop system 
\eqref{eq:closed_loop} on $[0,t_1]$,
and the solution $x$ satisfies
\[
\|x(t)\|_{\cl} < M\mu_0 \left(
e^{-ct}
\left(
1 - \frac{\alpha  \sigma}{c}
\right) + 
\frac{\alpha  \sigma}{c}
\right) 
\]
for all $t \in [0,t_1]$.
In particular, we have $x(t_1) \in \bcl(M\mu_1)$. 
Repeating this argument shows that there exists a unique solution $x$ of the closed-loop system \eqref{eq:closed_loop} on $\mathbb{R}_{\geq 0}$.
Moreover, 
the inequality \eqref{eq:state_decrease_TV_kth} holds for all $t \in [t_k,t_{k+1}]$ and $k \in \mathbb{Z}_{\geq 0}$.
Since Lemma~\ref{lem:decay_bound} shows that 
the constant $\gamma$ defined by \eqref{eq:bound_decay_rate_zo}
satisfies
\[
e^{-c \tau }
\left(
1 - \frac{\alpha  \sigma}{c}
\right) + 
\frac{\alpha  \sigma}{c} \leq e^{-\gamma \tau}
\]
for all $\tau \in [0, \tau_{\max}]$,
we obtain
\[
\|x(t)\|_{\cl} <
M\mu_k  e^{-\gamma (t-t_k)}  \leq 
M\mu_0  e^{-\gamma t} 
\]
for all $t \in [t_k, t_{k+1}]$ and $k \in \mathbb{Z}_{\geq 0}$.
\hspace*{\fill} $\blacksquare$

\section{Lur'e system}
\label{sec:lure}
Consider the Lur'e system
\begin{equation}
	\label{eq:Lure}
	\dot x(t) = Ax(t) + Bu(t) + \xi \varphi\big(\eta^{\top }x(t)\big)
\end{equation}
for $t \in \mathbb{R}_{\geq 0}$ with the initial state
$x(0) = x_0 \in \mathbb{R}^n$,
where $A \in \mathbb{R}^{n \times n}$, $B \in \mathbb{R}^{n\times m}$, 
$\xi, \eta \in \mathbb{R}^n \setminus \{0 \}$, and 
$\varphi\colon \mathbb{R} \to \mathbb{R}$ is continuously differentiable
such that $\varphi(0) = 0$.
In the absence of  quantization and self-triggered sampling,
the control input $u$ is given by
\begin{equation}
	\label{eq:linear_feedback}
	u(t) = Kx(t)
\end{equation}
for $t \in \mathbb{R}_{\geq 0}$,
where $K \in \mathbb{R}^{m\times n}$.
To apply the proposed methods, here we show how to check whether
Assumptions~\ref{assump:closed}--\ref{assump:control_part}
are satisfied for the Lur'e system above.
In this and the next section,
we choose the norms $\|\cdot\|_{\cl}$ and $\|\cdot\|_{\op}$
from the class of diagonally-weighted $\infty$-norms.

Let $R_0 \in  \mathbb{R}_{>0} \cup \{\infty \}$, and set
\begin{equation}
	\label{eq:kappa_bound}
	\kappa_{\min} \coloneqq \inf_{-R_0 < z < R_0} \varphi'(z),\quad 
	\kappa_{\max} \coloneqq \sup_{-R_0 < z < R_0} \varphi'(z).
\end{equation}
Define an open convex set $C_{\eta}$ by
\[
C_{\eta} \coloneqq \{x \in \mathbb{R}^n : -R_0 < \eta^{\top} x < R_0 \}.
\]
For $x\in \mathbb{R}^n$, 
define the functions $F_0$ and $f_q$ by
\begin{align*}
	F_0(x) &\coloneqq (A+BK)x + \xi \varphi(\eta^{\top }x)
\end{align*}
and
\begin{align*}
	f_q(x) &\coloneqq Ax +  \xi \varphi(\eta^{\top }x) + BKq,
\end{align*}
where $q\in \mathbb{R}^n$. 

\subsubsection*{On Assumption~\ref{assump:closed}}
Using the technique developed in 
\cite[Theorem~29]{Davydov2025}, we can find
a constant $c \in \mathbb{R}_{>0}$ and a diagonally-weighted $\infty$-norm $\|\cdot\|_{\cl}$
satisfying
\begin{equation}
	\label{eq:contraction_cond}
	\sup_{x \in C_{\eta} }
	\mu_{\cl}\big(DF_0(x)\big) \leq -c.
\end{equation}
In fact,
the inequality \eqref{eq:contraction_cond} holds if and only if
\begin{subequations}
	\label{eq:LP_for_c}
	\begin{align}
		\lceil (A+BK)+ \kappa_{\min} \xi \eta^{\top}\rceil_{\mathrm{Mzr}} \,\theta_{\cl}^{-1} &\leq -c\, \theta_{\cl}^{-1} \quad \text{and} 
		\label{eq:LP_for_c1}\\
		\lceil (A+BK)+ \kappa_{\max} \xi \eta^{\top}\rceil_{\mathrm{Mzr}} \, \theta_{\cl}^{-1} &\leq -c \, \theta_{\cl}^{-1}
		\label{eq:LP_for_c2}
	\end{align}
\end{subequations}
are satisfied,
where $\theta_{\cl}\in \mathbb{R}^n_{>0}$ is the weighting
vector for the norm $\|\cdot\|_{\cl}= \|\cdot \|_{\infty, [\theta_{\cl}]}$.
For a fixed $c\in \mathbb{R}_{>0}$, the problem of 
finding $\theta_{\cl} \in \mathbb{R}^n_{>0}$ satisfying 
the above two inequalities \eqref{eq:LP_for_c1} and 
\eqref{eq:LP_for_c2} can be solved by linear programming.
Since
\begin{align}
	\label{eq:Kx_bound}
	|\eta^{\top} x| \leq 
	\|\eta\|_{1,[\theta]^{-1} }  \|x\|_{\infty,[\theta]} ,\quad x \in \mathbb{R}^n,~
	\theta \in \mathbb{R}^n_{>0},
\end{align}
the maximum constant $R_1$ satisfying
$\bcl(R_1) \subseteq C_{\eta}$ is given by
\begin{equation}
	\label{eq:R1_cond_Lure}
	R_1 = \frac{R_0}{\|\eta\|_{1, [\theta_{\cl}]^{-1} } }.
\end{equation}

\subsubsection*{On Assumption~\ref{assump:open}}
As above, we can find
a constant $d_1 \in \mathbb{R}_{\geq 0}$ 
and a diagonally-weighted $\infty$-norm $\|\cdot\|_{\op}$
satisfying
\[
\sup_{x \in C_{\eta} }
\mu_{\op}\big(Df_q(x)\big) 
\leq d_1
\]
for all $q \in \mathbb{R}^n$, by solving the inequities 
\begin{subequations}
	\label{eq:LP_for_d}
	\begin{align}
		\lceil A+ \kappa_{\min} \xi \eta^{\top}\rceil_{\mathrm{Mzr}} \,\theta_{\op}^{-1} &\leq d_1  \theta_{\op}^{-1}  \quad  \text{and} 
		\label{eq:LP_for_d1} \\
		\lceil A+ \kappa_{\max} \xi \eta^{\top}\rceil_{\mathrm{Mzr}} \,\theta_{\op}^{-1} &\leq d_1 \theta_{\op}^{-1},
		\label{eq:LP_for_d2} 
	\end{align}
\end{subequations}
where $\theta_{\op}\in \mathbb{R}^n_{>0}$ is the weighting
vector for the norm $\|\cdot\|_{\op}= \|\cdot \|_{\infty, [\theta_{\op}]}$.
Moreover,
we have from the inequality \eqref{eq:Kx_bound} that 
\begin{align*}
	\|F_0(x) \|_{\op}
	&\leq \|A+BK\|_{\op} \|x\|_{\op}+ \kappa \|\xi\|_{\op}\|\eta\|_{1,[\theta_{\op}]^{-1} } \|x\|_{\op}
\end{align*}
for all $x \in C_{\eta}$, 
where 
$
\kappa \coloneqq \max\{|\kappa_{\min}|,\, |\kappa_{\max}|  \}.
$
Hence, if 
\begin{equation}
	\label{eq:Lure_d2}
	d_2 = \|A+BK\|_{\op} + \kappa \|\xi\|_{\op}\|\eta\|_{1,[\theta_{\op}]^{-1} },
\end{equation}
then
$\|F_0(x) \|_{\op} \leq d_2 \|x\|_{\op}$ for all $x \in C_{\eta}$.
As in \eqref{eq:R1_cond_Lure} for $R_1$,
the maximum constant $R_2 $ satisfying
$\bop(R_2) \subseteq C_{\eta}$ is given by
\begin{equation}
	\label{eq:R2_cond_Lure}
	R_2 = \frac{R_0}{\|\eta\|_{1,[\theta_{\op}]^{-1} }}.
\end{equation}

\subsubsection*{On Assumption~\ref{assump:control_part}}
To obtain $R \coloneqq \min \{R_1,\, R_2 / \Gamma \}$, 
we need a constant $\Gamma \in \mathbb{R}_{>0}$ 
satisfying  $\|x\|_{\op} \leq \Gamma\|x\|_{\cl}$.
The minimum constant $\Gamma$ satisfying this inequality
is given by
\begin{equation}
	\label{eq:mim_Gam}
	\Gamma =\max_{i=1,\dots,n} \frac{\theta_{\op,i}}{\theta_{\cl,i}}.
\end{equation}
The inequality \eqref{eq:F_e_Lip} is equivalent to
\[
\|BKe\|_{\cl} \leq \alpha \|e\|_{\op},\quad e \in \mathbb{R}^n.
\]
Then $\sigma_0 \in \mathbb{R}_{>0}$ can be chosen arbitrarily, and
the minimum constant $\alpha$ satisfying
the above inequality is given by
\begin{equation}
	\label{eq:alpha_Lure}
	\alpha = \big\| [\theta_{\cl}] BK [\theta_{\op}]^{-1} \big\|_{\infty}.
\end{equation}

\section{Numerical simulations of a two-tank system}
\label{sec:simulation}
In this section, we apply the proposed self-triggered control schemes to
a two-tank system.
First, we find constants satisfying
Assumption~\ref{assump:closed}--\ref{assump:control_part} by
using the results in Section~\ref{sec:lure}.
Next, we present simulation results of the STMs~\eqref{eq:STM}
and \eqref{eq:STM_tm}.
Finally, we 
make a comparison between 
the logarithmic and zooming quantization cases
with respect to 
the relative error of the state from the ideal closed-loop
system.

\subsection{Dynamics of a two-tank system}
Let a function $ \Phi\colon \mathbb{R} \to \mathbb{R}$ be given, and
consider the following system
\begin{equation}
	\label{eq:tank_system}
	\left\{
	\begin{aligned}
		\dot x_1(t) &= \Phi\big(x_2(t)- x_1(t)\big) + u(t) \\
		\dot x_2(t) &= -\Phi\big(x_2(t)- x_1(t)\big)
	\end{aligned}
	\right.
\end{equation}
for $t \in \mathbb{R}_{\geq 0}$ with the initial state
\[
\begin{bmatrix}
	x_1(0) \\ x_2(0)
\end{bmatrix}
=
\begin{bmatrix}
	x_{1,0}  \\ x_{2,0} 
\end{bmatrix} \in \mathbb{R}^2
\]
This equation represents a 
system of two tanks with equal cross-sectional areas, connected via a pipe. For $i=1,2$,
let $H_i\in \mathbb{R}_{\geq 0}$ be the nominal liquid level
of tank~$i$.
The $i$-th element $x_i$ of the state  is the deviation from the nominal level in tank~$i$.
The input $u$ is the control flow to tank~$1$.

Assume that $H \coloneqq H_2 - H_1 >0$, and define 
\[
\Phi(z) \coloneqq a \sqrt{H+z} - a\sqrt{H},\quad z \geq -H, 
\]
where $a\in\mathbb{R}_{>0}$ is a flow constant of the pipe.
The term $a \sqrt{H+z}$ represents the flow from tank~$2$ 
to tank~$1$ when $z = x_2 - x_1$.
Here we assume that, in addition to  the control flow $u$, 
the constant flow $a\sqrt{H}$ exits tank~$1$ and 
enters tank~$2$ to maintain 
the nominal liquid level.
We set $a \coloneqq 2$ and  $H \coloneqq 1$ in this study.
Fig.~\ref{fig:two_tank} illustrates the two-tank system.

\begin{figure}[tb]
	\centering
	\includegraphics[width = 5cm]{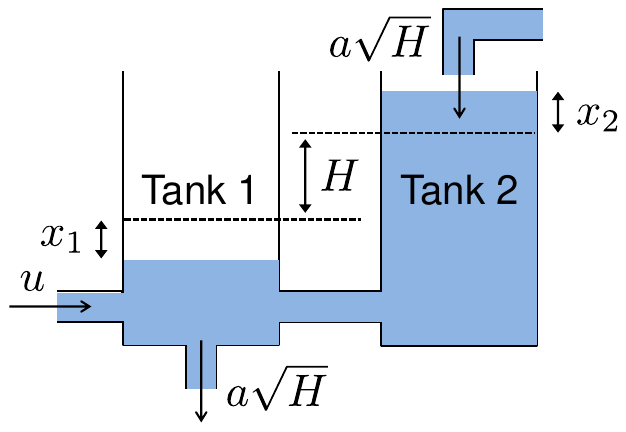}
	\caption{Two-tank system.}
	\label{fig:two_tank}
\end{figure}

Define $\varphi (z) \coloneqq  \Phi(z) - z$.
Then, the ODEs \eqref{eq:tank_system} can be rewritten  as
the Lur'e system \eqref{eq:Lure} with 
\[
A \coloneqq \begin{bmatrix}
	-1 & 1 \\
	1 & -1
\end{bmatrix},~~
B \coloneqq \begin{bmatrix}
	1 \\ 0
\end{bmatrix},~~
\xi \coloneqq 
\begin{bmatrix}
	1 \\ -1
\end{bmatrix},~~
\eta \coloneqq
\begin{bmatrix}
	-1 \\ 1
\end{bmatrix},
\]
where the eigenvalues of $A$ are $0$ and $-2$.
The feedback gain $K$ of the controller \eqref{eq:linear_feedback} is given by
\[
K \coloneqq
\begin{bmatrix}
	-0.7979 &   -0.6163
\end{bmatrix},
\]
which is the gain of the linear quadratic regulator with 
state weight $I$ and input weight $1$ for the linear part $(A,B)$.

By transforming the ODEs \eqref{eq:tank_system} into the
Lur'e system \eqref{eq:Lure}, we can check whether Assumptions~\ref{assump:closed}--\ref{assump:control_part}
are satisfied for the two-tank system,
as explained in Section~\ref{sec:lure}.
For simulations, we set 
$R_0 \coloneqq 0.45$.
Then 
the constants $\kappa_{\min}$ and $\kappa_{\max}$ in \eqref{eq:kappa_bound} 
are given by
$\kappa_{\min} = -0.17$ and $\kappa_{\max} = 0.35$.
Solving the inequalities \eqref{eq:LP_for_c1} and 
\eqref{eq:LP_for_c2} with $c = 0.4$,
we obtain the vector $\theta_{\cl} = [1 ~~0.5180]^{\top}$.
Similarly, we see that the vector $\theta_{\op} = [1 ~~1]^{\top}$
satisfies the inequalities  \eqref{eq:LP_for_d1} and \eqref{eq:LP_for_d2}  with $d_1 = 0$.
We set $\|\cdot\|_{\cl} \coloneqq \|\cdot\|_{\infty,[\theta_{\cl}]}$ and
$\|\cdot\|_{\op} \coloneqq \|\cdot\|_{\infty,[\theta_{\op}]}$.
Moreover, 
we have $R_1 = 0.1536$ by
\eqref{eq:R1_cond_Lure} and
$R_2 = 0.225$ by
\eqref{eq:R2_cond_Lure}. 
Using \eqref{eq:Lure_d2}, we calculate 
the constant $d_2$ to be $2.8817$.
From \eqref{eq:mim_Gam}, we obtain $\Gamma = 1.9305$. Then  
\[
R \coloneqq \min\left\{R_1,\,
\frac{R_2}{\Gamma} \right\} =  0.1166.
\] 
We have $\alpha = 1.4142$ by \eqref{eq:alpha_Lure}.
Note that the constant $\sigma_0 \in \mathbb{R}_{>0}$ 
in Assumption~\ref{assump:control_part} can be chosen arbitrarily.
We summarize the parameters for 
Assumptions~\ref{assump:closed}--\ref{assump:control_part}
in Table~\ref{tab:assumption_parameters}.
\begin{remark}
	If we set $\|\cdot\|_{\op} \coloneqq \|\cdot\|_{\cl}$,
	then $d_1 \geq 1.2562$, which leads to frequent sampling; see the definitions
	of $\psi_{\log}$ and $\psi_{\text{zo}}$ given in \eqref{eq:psi_log} and \eqref{eq:psi_dyn}, respectively.
	For this reason, we use different norms in Assumptions~\ref{assump:closed} and
	\ref{assump:open}.
	\hspace*{\fill} $\triangle$ 
\end{remark}

\begin{table}[tb]
	\centering
	\caption{Parameters for Assumptions~\ref{assump:closed}--\ref{assump:control_part}.}
	\label{tab:assumption_parameters}
	\begin{tabular}{c|cc||c|c} \hline 
		\textbf{Parameter} & \textbf{Value} && \textbf{Parameter} & \textbf{Value} \\ \hline
		$c$ & $0.4$ && $(d_1,d_2)$ & $(0,2.8817)$ 	 \\
		$\theta_{\cl}$ & $[1 ~~0.5180]^{\top}$  && $\theta_{\op}$ & $[1 ~~1]^{\top}$  \\
		$R_1$ & $0.1536$ && $R_2$ & $0.225$ 	 \\
		$\Gamma$ & $1.9305$ &&	$R$ & $0.1166$    \\
		$\alpha$ & $1.4142$ &&  $\sigma_0$ & arbitrary	\\
		\hline	
	\end{tabular}
\end{table}

\subsection{Simulation results of self-triggered control under logarithmic quantization}
First, we consider the STM \eqref{eq:STM} for logarithmic quantization.
The constants $L_{\cl}$ and $L_{\op}$ in \eqref{eq:Lcl_Lop_def} 
are given by $L_{\cl} = L_{\op} = 1$.
Fig.~\ref{fig:sigma_cond} shows 
the range of the threshold parameter $\sigma$ such that 
the condition~\eqref{eq:thres_cond} is satisfied for a given
quantization density $\rho$.
The blue and the red lines indicate 
the lower bound $\Gamma  L_{\op}(1-\rho)/(1+\rho)$ and 
the upper bound $2c\rho / (\alpha L_{\cl}(1+\rho))$, respectively.
When the pair $(\rho,\sigma)$  belongs to the gray region, stabilization 
can be achieved by the STM \eqref{eq:STM}.
The upper and lower bounds intersect and take the value $0.2467$ at 
\[
\rho = \rho_{\min} \coloneqq 0.7734.
\]
Therefore, when $\rho > \rho_{\min}$,
the state trajectory starting in $\bcl(R/L_{\cl})$ exponentially converges to the origin
without Zeno behavior
for a suitable $\sigma$ by Theorem~\ref{thm:log_case}. On the other hand,
the upper bound $2c\rho / (\alpha L_{\cl}(1+\rho))$ is  $0.2828$
at $\rho = 1$. Hence,
the state convergence is not guaranteed for any $\sigma \geq 0.2828$ under the STM \eqref{eq:STM}.

\begin{figure}[t]
	\centering
	\includegraphics[width = 9cm]{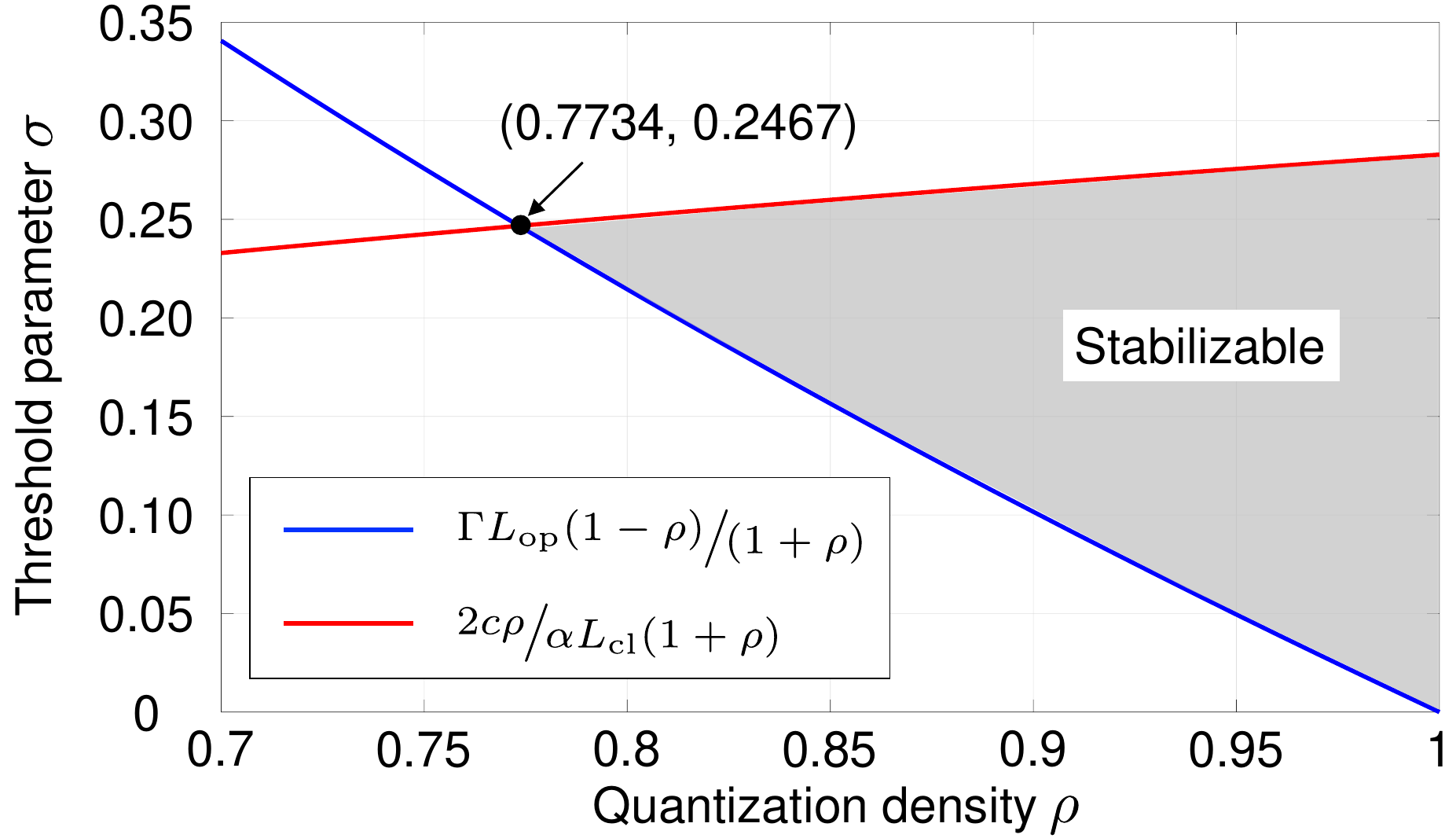}
	\caption{Stabilizable region of $(\rho, \sigma)$.}
	\label{fig:sigma_cond}
\end{figure}

For simulations of time-responses,
the parameters of logarithmic quantization are given by
$\chi_{0} = R = 0.1166$ and $\rho = 0.85$. 
Then Assumption~\ref{assump:initial_cond}
is satisfied.
The condition \eqref{eq:thres_cond} on 
the threshold parameter $\sigma$ 
is written as 
\[
0.1565 < \sigma < 0.2599,
\]
and  we set
$\sigma = 0.25$. 
The upper bound $\tau_{\max} \in \mathbb{R}_{>0}$ 
of the inter-sampling times is arbitrary, and 
we set $\tau_{\max} = 0.18$.
We see from \eqref{eq:min_inter_event_time} that 
the inter-sampling times of the STM
\eqref{eq:STM} are bounded from below by 
\[
\min\{ \tau_{\min},\, \widetilde \tau_{\min} \} = 
\min\{
0.0168,\, 0.0281
\} = 
0.0168,
\]
where 
$\widetilde \tau_{\min} = 0.0281$ is obtained by using
the constant $\lambda = 0.9251$ in Lemma~\ref{lem:Qlog_prop}.
Table~\ref{tab:log_parameters} summarizes the parameters used for the simulations
in the logarithmic quantization case.

\begin{table}[t]
	\centering
	\caption{Parameter settings for logarithmic quantization case.}
	\label{tab:log_parameters}
	\begin{tabular}{c|cc||c|c} \hline
		\textbf{Parameter} & \textbf{Value} && \textbf{Parameter} & \textbf{Value} \\ \hline 
		$\chi_0$ & $0.1166$ && 	$\rho$ & $0.85$  \\
		$\sigma$ & $0.25$  && 	$\tau_{\max}$ & $0.18$\\
		\hline
	\end{tabular}
\end{table}

Figs.~\ref{fig:state_log}--\ref{fig:ie_log} show the time-responses with 
the initial state 
$[x_{1,0}~~x_{2,0}]^{\top} = [0.1~~{-0.2}]^{\top}$, 
where the time-step of the simulation
is given by $10^{-5}$.
We illustrate the state $x$ in Fig.~\ref{fig:state_log}, the input $u$ in Fig.~\ref{fig:input_log}, and
the inter-sampling time $t_{k+1}-t_k$ in Fig.~\ref{fig:ie_log}. 
The total number of sampling instants on the interval $(0,6]$ is $98$.
In Fig.~\ref{fig:state_log},
the solid blue and red lines represent the first and second elements $x_1$ and $x_2$ of the  state of the
quantized self-triggered control system, respectively.
The dashed brawn and green lines are the first and second 
elements $x_{\text{ideal},1}$ and $x_{\text{ideal},2}$ of the state 
of the ideal closed-loop system, respectively, where the input $u_{\text{ideal}}$ is given by $u_{\text{ideal}}(t) = K[x_{\text{ideal},1}(t)~~x_{\text{ideal},2}(t)]^{\top}$ for all $t \in 
\mathbb{R}_{\geq 0}$.
In Fig.~\ref{fig:input_log}, the solid blue and dashed red lines correspond to
the input $u$ of the quantized self-triggered control system and the 
input $u_{\text{ideal}}$ of the ideal closed-loop system, respectively.

From Fig.~\ref{fig:state_log}, we see that the state trajectory of the quantized self-triggered 
control system looks identical to that of the ideal closed-loop system.
However, Fig.~\ref{fig:input_log} shows that the input oscillates 
on the interval $[0,1]$ due to quantization.
As shown in Fig.~\ref{fig:ie_log},
the inter-sampling time is smaller on the same interval, because
the state changes rapidly.
We also observe that the inter-sampling time remains almost constant
as the state converges.

\begin{figure}[t]
	\centering
	\includegraphics[width = 9cm]{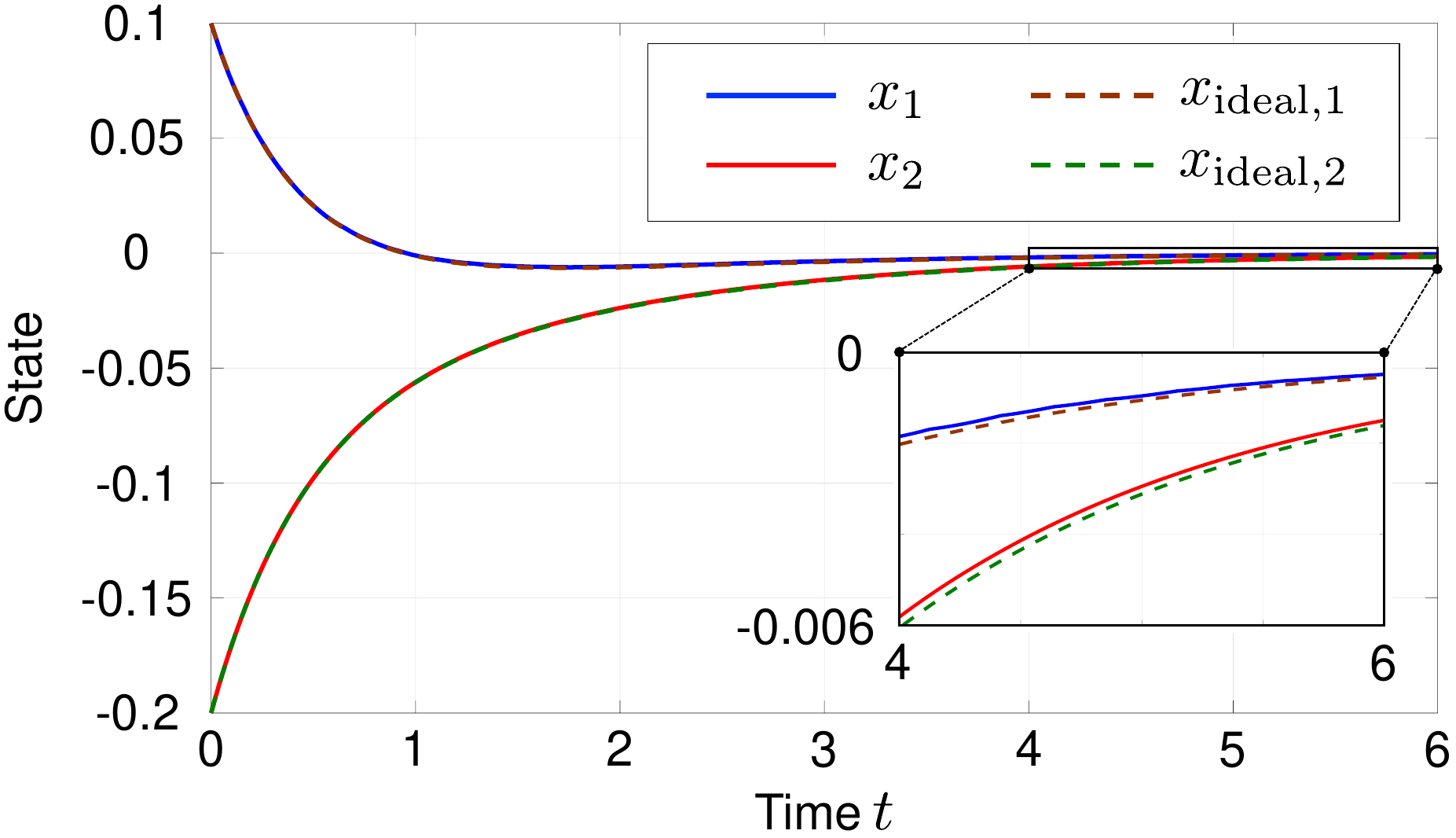}
	\caption{State under logarithmic quantization.}
	\label{fig:state_log}
\end{figure}

\begin{figure}[t]
	\centering
	\includegraphics[width = 9cm]{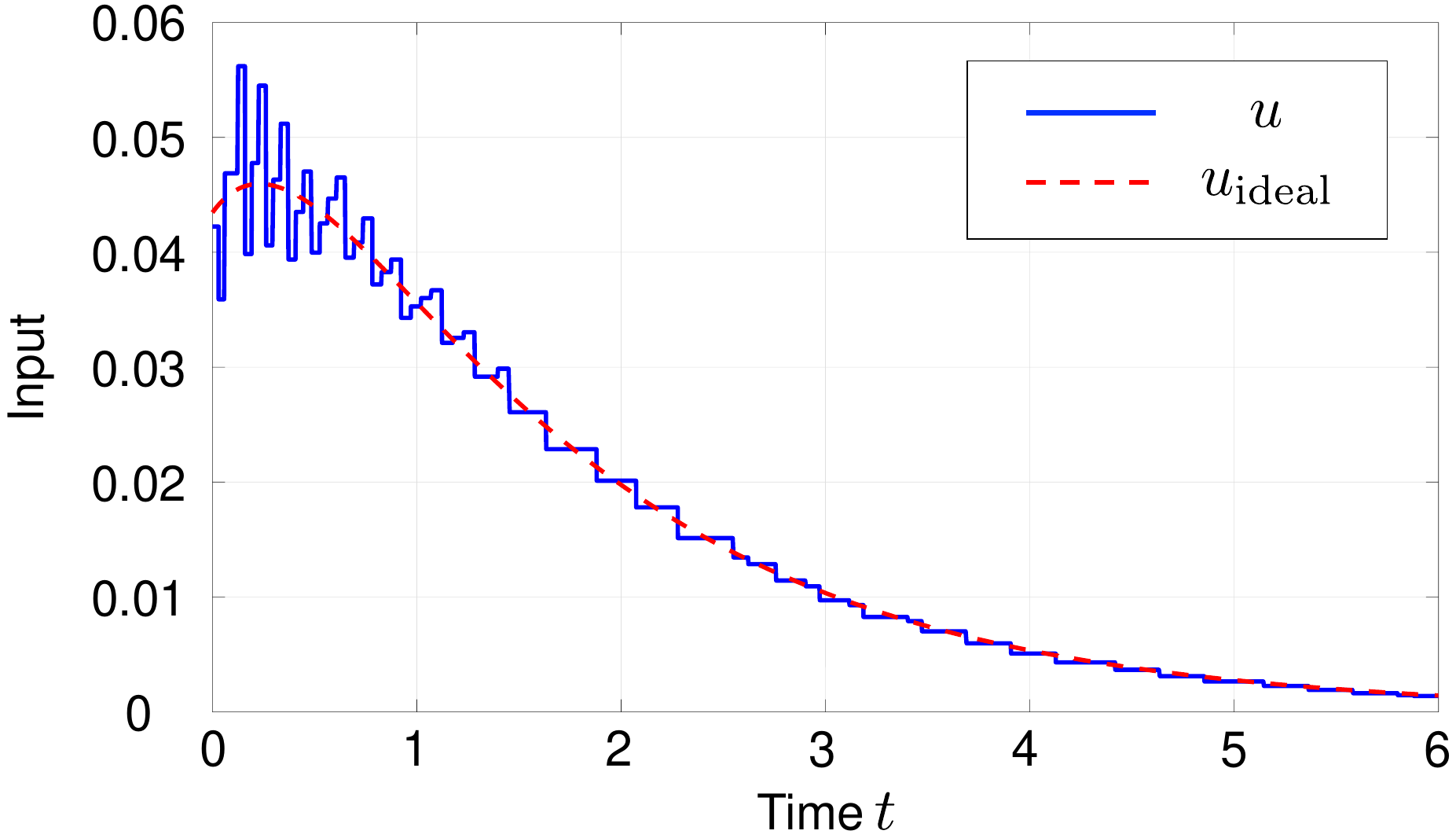}
	\caption{Input under logarithmic quantization.}
	\label{fig:input_log}
\end{figure}

\begin{figure}[t]
	\centering
	\includegraphics[width = 9cm]{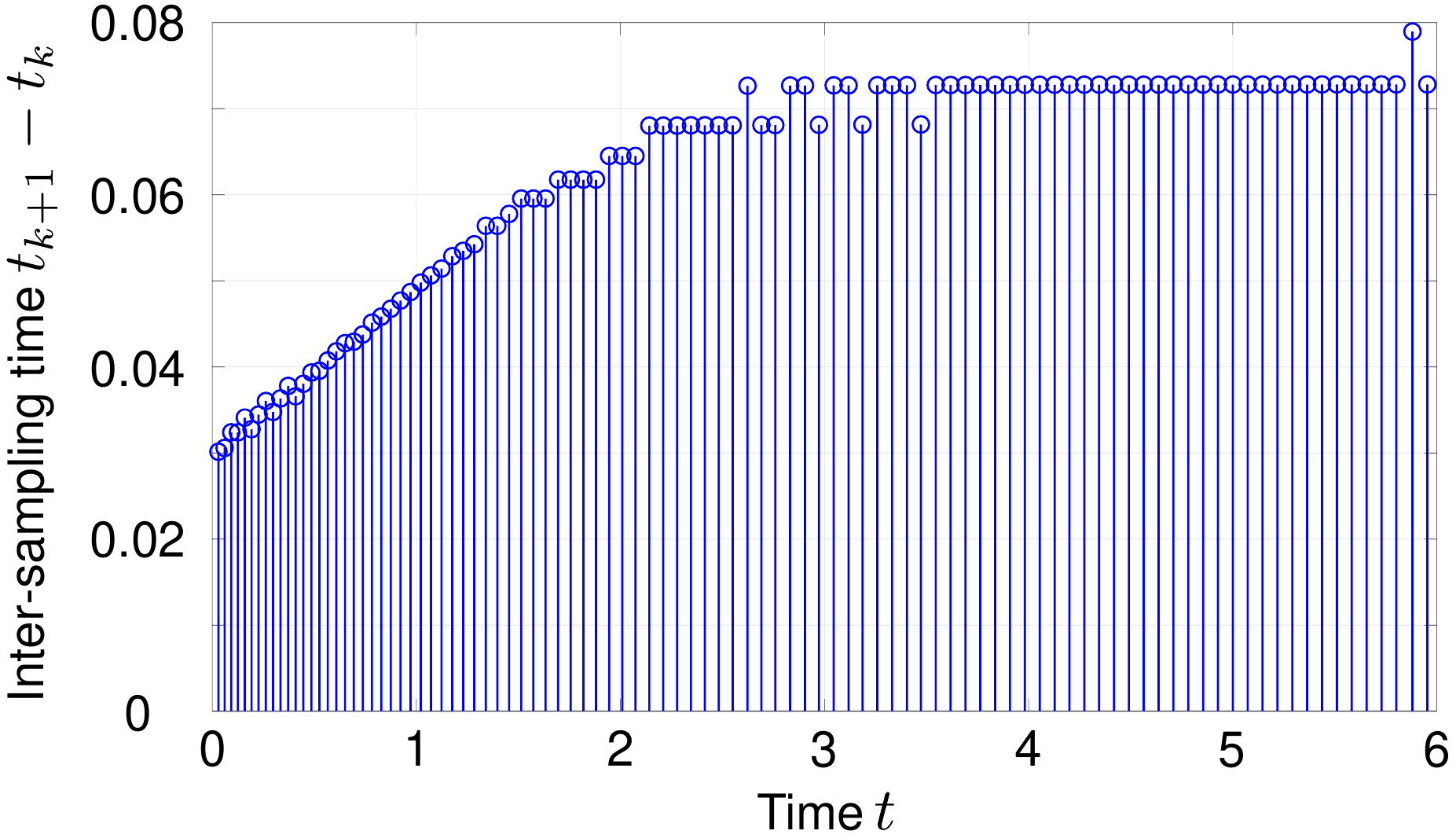}
	\caption{Inter-sampling times under logarithmic quantization.}
	\label{fig:ie_log}
\end{figure}

\subsection{Simulation results of self-triggered control under zooming quantization}
Next, we give simulation results in the case of zooming quantization.
For the $i$-th element $x_i$ of the state, 
the function $Q$ performs uniform quantization with a  
step size of $2\Delta / \theta_{\op,i}$,
where $\theta_{\op,i}$ is the $i$-th element of 
the weighting vector $\theta_{\op}$. In other words,
the quantized value of $x_i$ is given by the nearest value in the set 
$\{ 2p\Delta / \theta_{\op,i}: p \in \mathbb{Z}\}$.
The parameters of the zooming quantizer are given by
$M = 0.105$, $\Delta = 0.005$, and $\mu_0 = 1$.
Then Assumption~\ref{assump:initial}
is satisfied.
The period $h$ 
for the STM
\eqref{eq:STM_tm} has to satisfy
\[
0 < h \leq \widetilde \tau_{\min} = 0.0057,
\]
where $\lambda = 0.9837$ obtained from \eqref{eq:lambda_def_TV} is used
for the computation of $\widetilde \tau_{\min}$.
We set $h = 0.001$ for simulations.  The upper bound $\ell_{\max} \in \mathbb{N}$ 
of $\ell_k$ is given by $\ell_{\max} = \tau_{\max} /h = 
180$.
The condition \eqref{eq:thres_cond_varying} 
with this period $h$ is 
written as 
\[
0.0533 \leq \sigma < 0.2828,
\]
and 
we set $\sigma = 0.075$. 
The parameters used for the simulations
in the zooming quantization case are summarized in
Table~\ref{tab:zoom_parameters}.

Note that the comparison between the threshold parameters $\sigma$ of 
the STMs \eqref{eq:STM} and
\eqref{eq:STM_tm} does not make sense. In fact,
$\sigma$ is the coefficient of 
the quantized value $q_k$ in the logarithmic quantization case, while $\sigma$ is the coefficient 
of
the quantization range $M \mu_k$ in the zooming quantization case.
Moreover, a smaller $\sigma$  leads to
fast decay of the zoom parameter $\mu_k$ and hence quantization errors.
Therefore,
we choose the small threshold in the zooming  quantization case.

\begin{table}[t]
	\centering
	\caption{Parameter settings for zooming quantization case.}
	\label{tab:zoom_parameters}
	\begin{tabular}{c|cc||c|c} \hline
		\textbf{Parameter} & \textbf{Value} && \textbf{Parameter} & \textbf{Value} \\ \hline 
		$M$ & $0.105$ && 	$\Delta$ & $0.005$  \\
		$\mu_0$ & $1$  && 	$h$  &  $0.001$ \\
		$\ell_{\max}$ & $180$  && 	$\sigma$ & $0.075$  \\
		\hline	
	\end{tabular}
\end{table}

The time-responses with the initial state $[x_{1,0}~~x_{2,0}]^{\top} = [0.1~~{-0.2}]^{\top}$ are shown in Figs.~\ref{fig:state_zo}--\ref{fig:ie_zo}, where the time-step of the simulation is set to $10^{-5}$.
We illustrate the state $x$ in Fig.~\ref{fig:state_zo} and the input $u$ in 
Fig.~\ref{fig:input_zo}, where each line represents the same as in 
Figs.~\ref{fig:state_log} and \ref{fig:input_log}.  Fig.~\ref{fig:ie_zo}
shows
the inter-sampling time $t_{k+1}-t_k$. 
The total number of sampling instants 
on the interval $(0,6]$ is $100$, which is closely aligns with 
the number of sampling instants, $98$, in 
the logarithmic quantization case.
From Figs.~\ref{fig:state_zo}--\ref{fig:ie_zo}, we observe that the responses in
the zooming  quantization case have properties similar to those in
the logarithmic quantization case.
In fact, the state trajectory by quantized self-triggered control
is quite close to that by the ideal control, and
the input by quantized self-triggered  control oscillates in  the early response phase.

Compared with the case of logarithmic quantization,
the change of the input $u$ in Fig.~\ref{fig:input_zo} is still large on the interval $[4,6]$ due to 
coarse quantization.
From Fig.~\ref{fig:state_zo}, we observe that 
this leads to the oscillation of the state $x_1$ on the interval $[4,6]$.
In Fig.~\ref{fig:ie_zo}, the
inter-sampling time is smaller on the interval $[0,1]$ but larger on the interval
$[2,6]$ than that in the logarithmic quantization case.
This is because the STM~\eqref{eq:STM_tm} for zooming quantization uses
the quantization range
$M\mu_k$ for the threshold.
To avoid quantizer saturation, the update rule for the zoom parameter $\mu_k$
is designed to be conservative. In other words,
$\mu_k$ decreases more slowly than $x$.
This makes the inter-sampling time on the interval $[2,6]$ larger  than
in the case of logarithmic quantization.

\begin{figure}[t]
	\centering
	\includegraphics[width = 9cm]{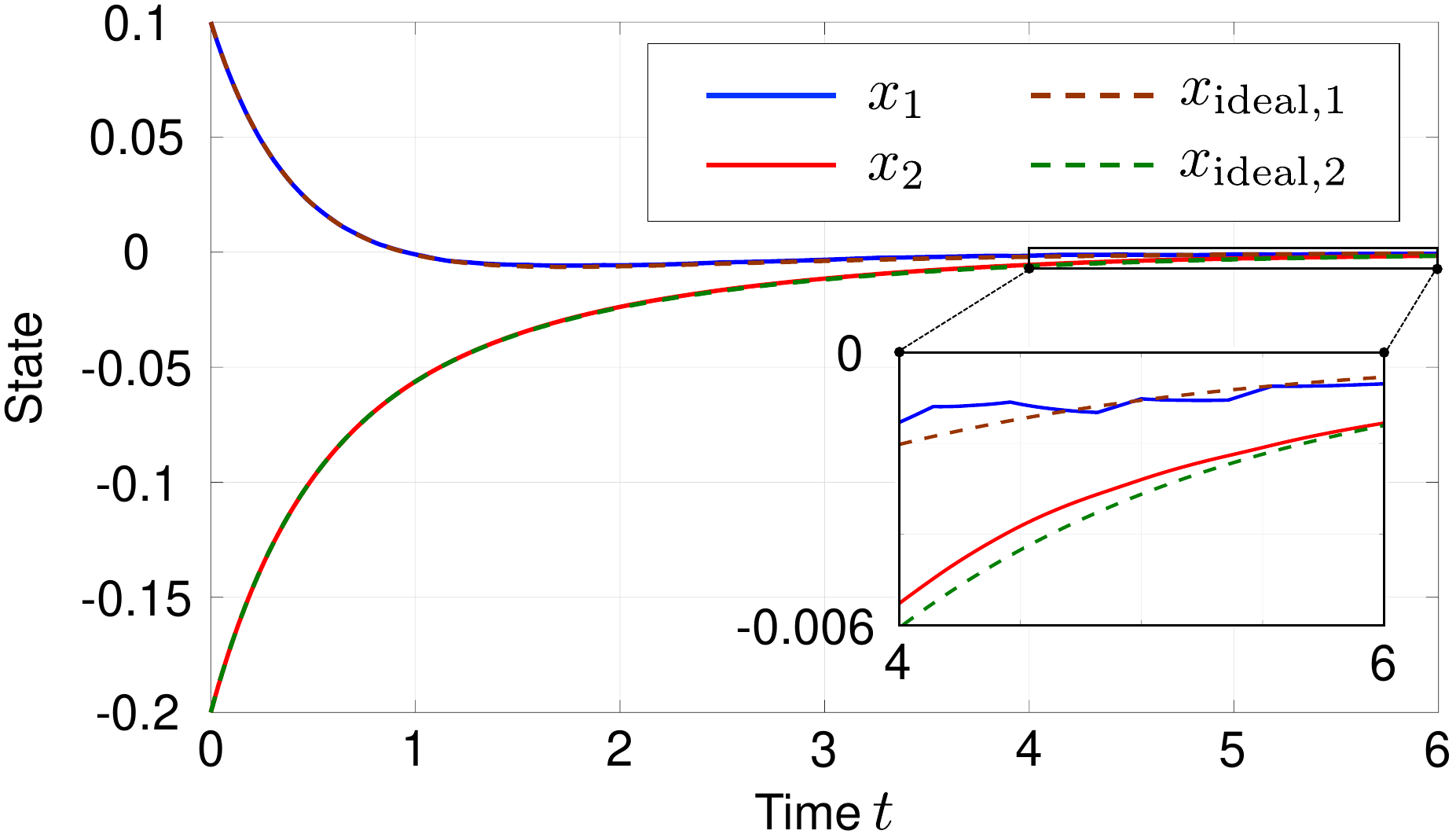}
	\caption{State under zooming quantization.}
	\label{fig:state_zo}
\end{figure}

\begin{figure}[t]
	\centering
	\includegraphics[width = 9cm]{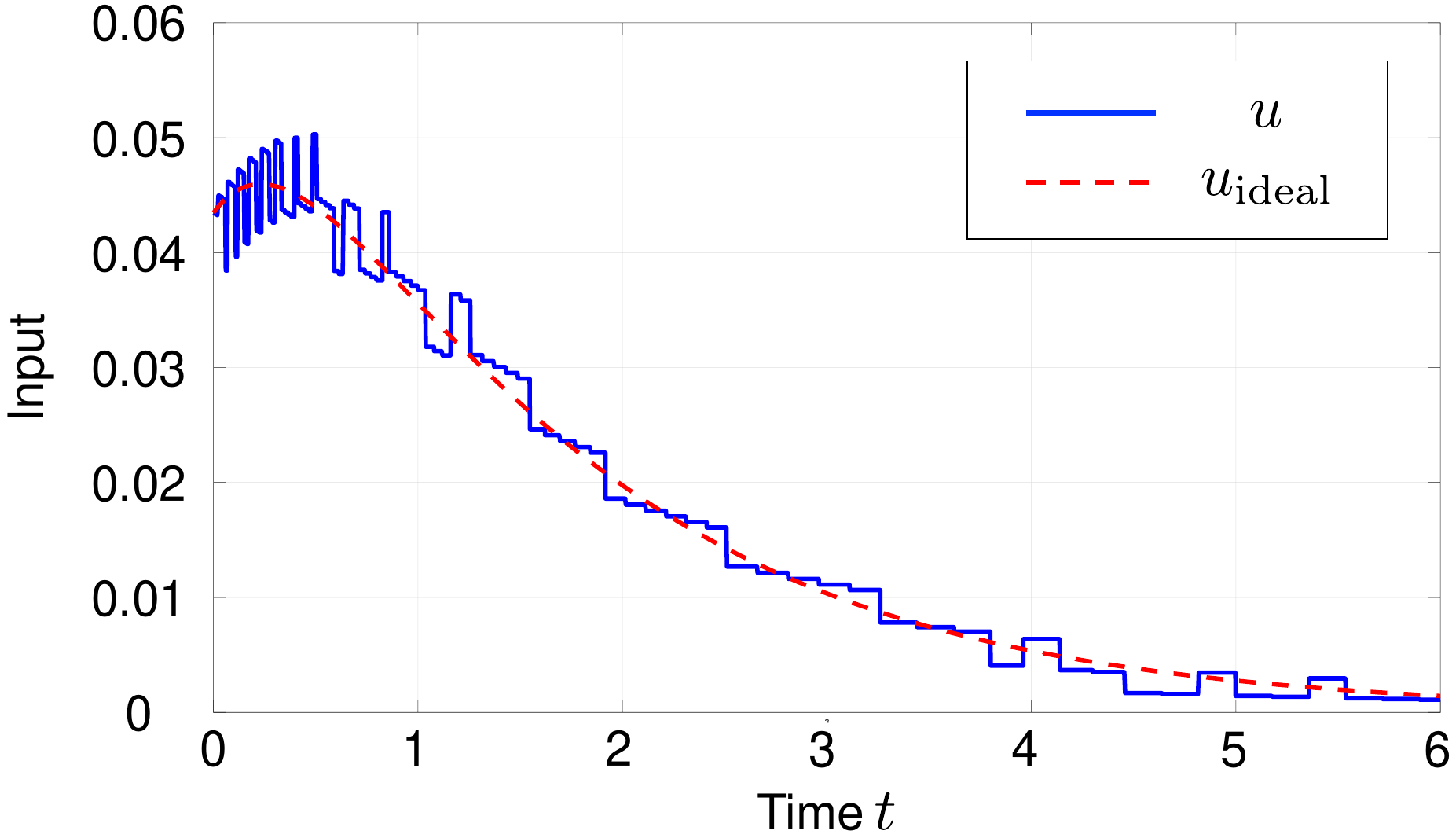}
	\caption{Input under zooming quantization.}
	\label{fig:input_zo}
\end{figure}

\begin{figure}[t]
	\centering
	\includegraphics[width = 9cm]{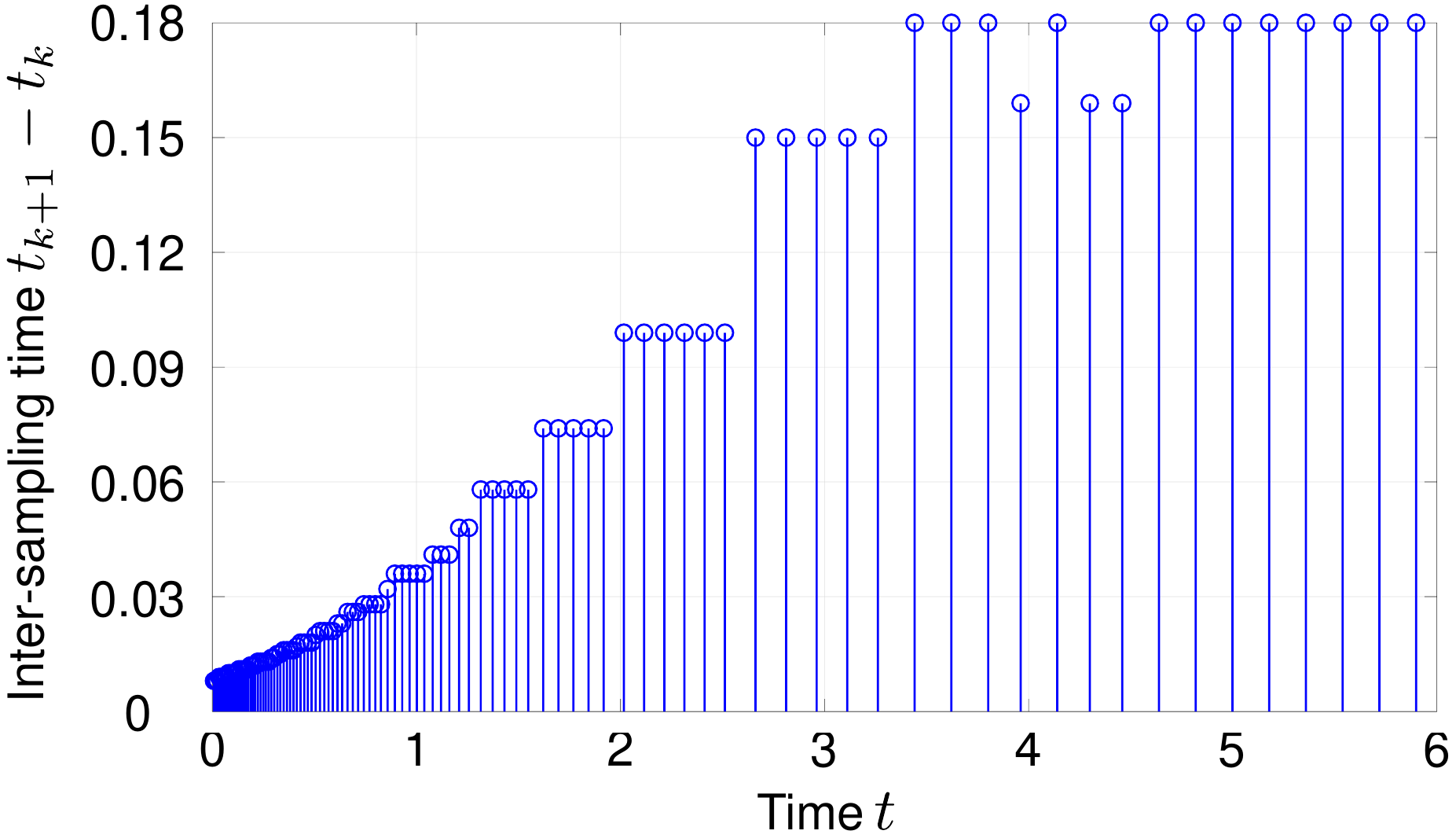}
	\caption{Inter-sampling times under zooming quantization.}
	\label{fig:ie_zo}
\end{figure}

\begin{figure}[t]
	\centering
	\includegraphics[width = 9cm]{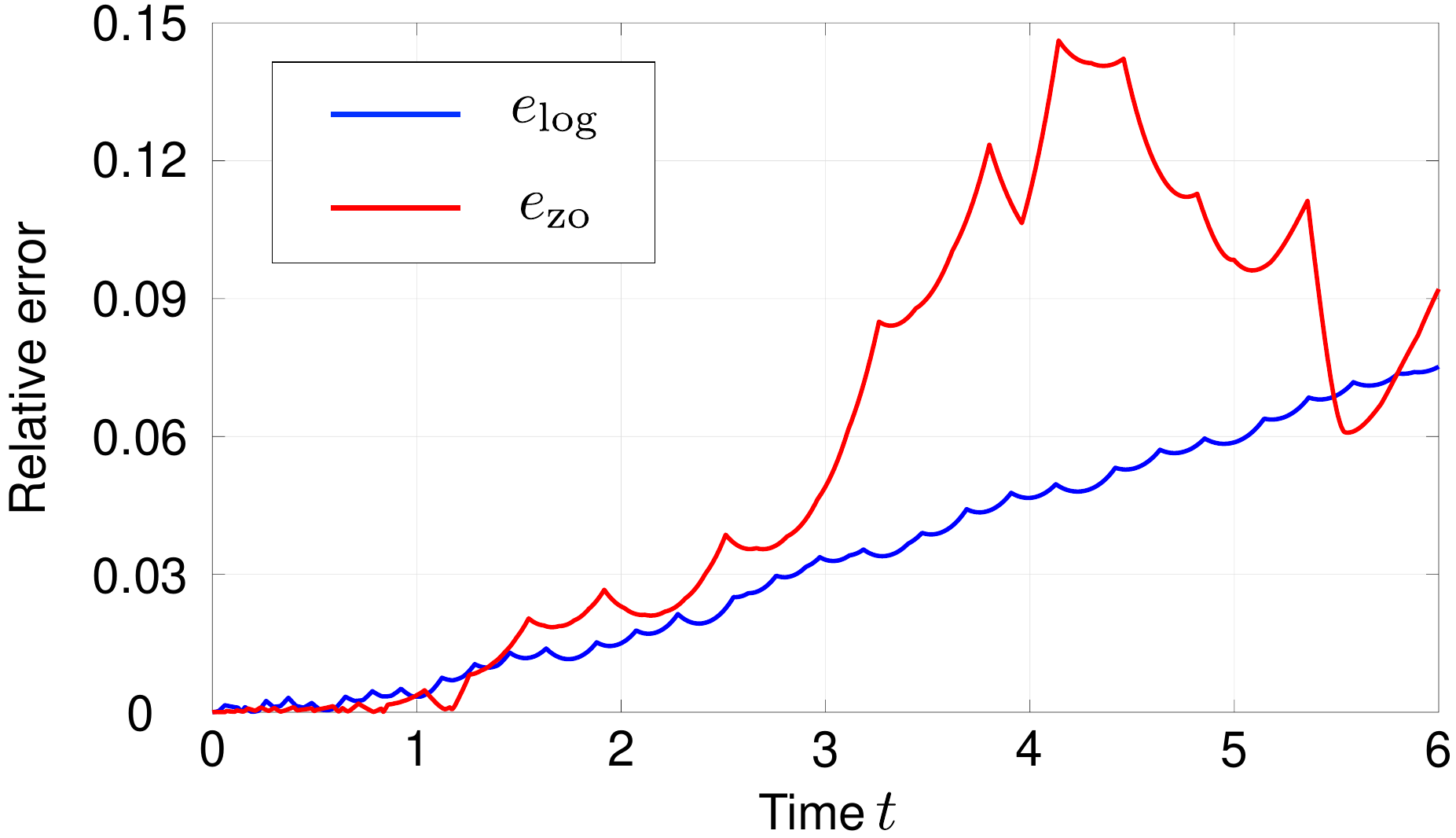}
	\caption{Comparison of relative errors.}
	\label{fig:relative_error}
\end{figure}

\subsection{Comparison of relative errors}
In Fig.~\ref{fig:relative_error}, we 
compare the cases of logarithmic quantization and zooming quantization
with respect to the relative error of the state.
Let $x_{\log,1}$ and $x_{\log,2}$ denote the first and 
second elements of the state of the self-triggered control system
with logarithmic quantization, whose parameters are as in Table~\ref{tab:log_parameters}. Then the relative error of the state under 
logarithmic quantization is defined by
\begin{align*}
	e_{\log}(t)
	\coloneqq \frac{
		\sqrt{
			\sum_{i=1}^2
			|x_{\log,i}(t) - x_{\text{ideal},i}(t)|^2
	} }
	{\sqrt{
			\sum_{i=1}^2
			| x_{\text{ideal},i}(t)|^2
	}}
\end{align*}
for $t \geq 0$. The relative error $e_{\text{zo}}$ is defined in the same way
for the state under zooming quantization, where the parameters are  as
in Table~\ref{tab:zoom_parameters}.
In Fig.~\ref{fig:relative_error},
the blue and red lines correspond to $e_{\log}$
and $e_{\text{zo}}$, respectively.
We see that 
$e_{\log}$ grows linearly, while
$e_{\text{zo}}$ increases rapidly
on the interval $[3,4]$. 
Under logarithmic quantization, the quantization error is bounded by a constant multiple of the state norm; see the inequalities \eqref{eq:log_prop2} and \eqref{eq:log_prop1}.
On the other hand, the zoom parameter $\mu_k$
is updated such that its exponential rate of
decay is smaller than that of the state. As a result,
the relative quantization error may increase exponentially
in the zooming quantization case.
Thus, $e_{\text{zo}}$  has
a growth behavior different from 
$e_{\log}$.

\section{Conclusion}
\label{sec:conclusion}
We studied the problem of stabilizing nonlinear systems under quantization and 
self-triggered sampling.
In the closed-loop system we consider,  only the quantized data of the state
are available to the controller and the STM.
The key assumption for stabilization 
is that the closed-loop system is contracting 
in the ideal case without quantization or self-triggered sampling.
First, we presented a self-triggered control scheme that establishes,
under logarithmic quantization, the exponential
convergence of all state trajectories starting in a given region.
Next, we proposed a co-design method for zooming quantization and self-triggered 
sampling to achieve stabilization. In both cases of logarithmic quantization and zooming quantization,
the proposed STMs estimate the measurement error by
predicting the future state trajectory
from the quantized state.
We also discussed the assumptions of the proposed methods 
for Lur'e systems.

There are still some open problems in quantized self-triggered control for 
nonlinear systems.
Since the system model is used for state prediction,
the next sampling time cannot be computed correctly in the presence of 
disturbances and model uncertainties. 
Therefore, it would be beneficial to
extend the techniques presented here to
uncertain systems with disturbances. In applications to networked control systems, 
it is also important to explicitly 
address the robustness of the proposed control schemes with respect to transmission delays.
Other future research directions include finding norms that
lead to the optimal choice of parameters for quantization and self-triggered
sampling.

\appendix
\noindent\hspace{1em}{\textit{Proof of Lemma~\ref{lem:decay_bound}.} }
Since 
$
w(t) <  1
$
for all $t \in \mathbb{R}_{>0}$,
the definition of $W$ gives 
$
\gamma = W(\tau_{\max}) > 0.
$

To prove the inequality \eqref{eq:f_tmax_bound},
we show that the function $W$ 
is monotonically decreasing on $\mathbb{R}_{>0}$.
Let $t \in \mathbb{R}_{>0}$. 
Since
\[
W'(t) = \frac{1}{t^2}
\left( - \frac{w'(t)}{w(t)}t +  \log w(t)   \right),
\]
it follows that 
\begin{equation}
	\label{eq:app_equiv1}
	W'(t) < 0 \quad \Leftrightarrow \quad \Lambda(t) := -\frac{w'(t)}{w(t)} t + \log w(t) < 0.
\end{equation}
Moreover,
\[
\Lambda'(t) = \frac{t\big(w'(t)^2 - w(t) w''(t)\big)}{w(t)^2},
\]
and hence
\begin{equation}
	\label{eq:app_equiv2}
	\Lambda'(t) < 0 \quad \Leftrightarrow \quad w'(t)^2 - w(t) w''(t) <0.
\end{equation}
We have
\[
w'(t) = -c(1-\varepsilon ) e^{-c t}\quad  \text{and} \quad 
w''(t) = c^2(1-\varepsilon ) e^{-c t}.
\]
From $0 < \varepsilon < 1$, we obtain
\[
w'(t)^2 - w(t) w''(t) = -\varepsilon (1-\varepsilon )c^2 e^{-2c t} <0,
\]
and therefore $\Lambda'(t) < 0$ by \eqref{eq:app_equiv2}.
Since $w(0) = 1$, we have
$\Lambda(0) = 0.
$
This yields $\Lambda(t) < 0$.
Hence, $W$ is monotonically decreasing on $\mathbb{R}_{>0}$ by
\eqref{eq:app_equiv1}.

From the monotonic decreasing property of $W$, we obtain
$W(t) \geq W(\tau_{\max})$ for all $t \in (0,\tau_{\max}]$,
which implies that
\[
\log w(t)\leq -\gamma t
\]
for all $t \in [0,\tau_{\max}]$. Thus, 
the inequality \eqref{eq:f_tmax_bound} holds.
\hspace*{\fill} $\blacksquare$

\end{document}